\documentclass[11pt,reqno]{amsart}

\usepackage{amssymb}
\usepackage{color,xcolor}
\usepackage{hyperref}
\usepackage{amsthm}
\usepackage{subcaption}
\usepackage{amsmath,amsfonts}
\usepackage{lineno,version}
\usepackage[utf8]{inputenc}

\usepackage{bm}
\usepackage{indentfirst,amsmath,amsfonts,amssymb,amsthm,cite}
\usepackage{lineno}
\usepackage{subeqnarray}
\usepackage{cases}
\usepackage{mathrsfs}

\usepackage{fancyhdr}
\usepackage{graphicx}
\usepackage{mathrsfs}
\usepackage{float}
\usepackage{color}
\usepackage{caption}
\newtheorem{theorem}{\hskip\parindent\bf Theorem}[section]

\newtheorem{lemma}{\hskip\parindent \bf Lemma}[section]
\newtheorem{definition}{\hskip\parindent\bf Definition }[section]
\newtheorem{remark}{\hskip\parindent\bf Remark}[section]

\def\bc{\begin{center}}
	\def\ec{\end{center}}
\numberwithin{equation}{section}{}

\def\bvarphi{\mbox{\boldmath$\varphi$}}
\newcommand{\ba}{\begin{array}}\newcommand{\ea}{\end{array}}
\newcommand{\be}{\begin{eqnarray}}\newcommand{\ee}{\end{eqnarray}}
\newcommand{\bex}{\begin{eqnarray*}}
	\newcommand{\eex}{\end{eqnarray*}}

\begin{document}
	
	\title[Second Order Reduced Model via Incremental Projection for Navier Stokes]
	{Second Order Reduced Model via Incremental Projection for Navier Stokes}
	
	\author{Y. Guo}
	\thanks{Y. Guo: School of Mathematical Sciences and Fujian Provincial Key Laboratory of Mathematical Modeling and High Performance Scientific Computing, Xiamen University, 361005 Xiamen, China. {\tt yayuguo@stu.xmu.edu.cn}}
	
	\author{C. N\'u\~nez Fern\'andez}
	\thanks{C. N\'u\~nez Fern\'andez: Departamento EDAN \& IMUS, Universidad de Sevilla, Spain. {\tt cnfernandez@us.es}}
	
	\author{S. Rubino}
	\thanks{S. Rubino: Departamento EDAN \& IMUS, Universidad de Sevilla, Spain. {\tt samuele@us.es}}
	
	\author{C. Xu}
	\thanks{C. Xu: School of Mathematical Sciences and Fujian Provincial Key Laboratory of Mathematical Modeling and High Performance Scientific Computing, Xiamen University, 361005 Xiamen, China. {\tt cjxu@xmu.edu.cn}}

	
	
	\date {\today}
	
	\begin{abstract}
		The numerical simulation of incompressible flows is challenging due to the tight coupling of velocity and pressure. Projection methods offer an effective solution by decoupling these variables, making them suitable for large-scale computations. This work focuses on reduced-order modeling using incremental projection schemes for the Stokes equations. We present both semi-discrete and fully discrete formulations, employing BDF2 in time and finite elements in space. A proper orthogonal decomposition (POD) approach is adopted to construct a reduced-order model for the Stokes problem. The method enables explicit computation of reduced velocity and pressure while preserving accuracy. We provide a detailed stability analysis and derive error estimates, showing second-order convergence in time. Numerical experiments are conducted to validate the theoretical results and demonstrate computational efficiency.
	\end{abstract}
	\maketitle
	\section{Introduction}
	Numerical simulation of incompressible flows presents a significant challenge due to the coupling of velocity and pressure imposed by the incompressibility constraint. Projection methods emerged in the late 1960s as a promising solution to this problem, pioneered by the groundbreaking work of Chorin and Temam \cite{Chorin,Temam}. These methods are particularly appealing because they allow the decoupling of velocity and pressure computations at each time step, reducing the problem to a series of independent elliptic equations. This decoupling makes projection methods highly efficient for large-scale numerical simulations. 
	
	Although projection methods fall under the broader category of fractional or splitting step approaches, the standard methodologies for fractional step methods (e.g., Yanenko \cite{Yanenko} or Glowinski [\cite{Glowinski}, Chap. II]) do not apply directly. This is because pressure is not a dynamic variable, and the Navier–Stokes equations do not conform to the Cauchy–Kovalevskaya form. Consequently, developing and analyzing higher-order projection methods has been a non-trivial task. Over the past  years, extensive research has been conducted to construct, analyze, and refine projection schemes, with considerable effort devoted to finding optimal approaches (see, e.g., Glowinski [\cite{Glowinski}, Chap. VII]). 
	
	Although many valuable insights have been gained, the field has also had its share of erroneous or misleading claims. However, recent advances have addressed several long-standing theoretical and practical challenges, leading to a clearer understanding of projection methods. Recognizing this progress, Guermond and collaborators (see \cite{guermond2006overview}) have compiled a comprehensive overview, presenting the best approximation results for various classes of projection schemes, and addressing critical implementation issues that are rarely discussed in the literature.
	Despite their advantages, relatively little research has focused on reduced-order methods using time-splitting techniques for unsteady Stokes or incompressible Navier-Stokes problems, contrary for instance to stabilized methods, see e.g. \cite{samu,novo2021error,rebollo2022error,garcia2024pod}. Notable exceptions include the work of Li et al. \cite{li2022efficient}, who proposed a reduced-order model (ROM) based on proper orthogonal decomposition (POD-ROM) for unsteady Stokes equations. Their approach combines the classical Chorin-Temam projection method with POD techniques, decoupling the reduced velocity and pressure variables. They circumvented the need to verify the classical LBB/inf-sup condition in mixed reduced spaces by using a stabilized pressure projection of the Petrov-Galerkin (PSPG) type. More recently, Azaiez et al. (see \cite{Azaiez2025}) have presented a reduced-order method based on the Goda time-splitting scheme (standard incremental) \cite{goda1979multistep}, an improvement over the classical Chorin-Temam pressure correction method. Their approach computes POD pressure modes using an inner product that not only matches the pressure regularity in full-order finite element (FE) solutions, but also allows fully explicit computation for reduced velocity and pressure. This framework improves both consistency and computational efficiency.
	
	In this paper, we propose to consider incremental projection methods to approximate the reduced-order Navier-Stokes equation (see \cite{guermondCRAS1997}). The advantage of this method is the following, the error on the velocity in the semi-discret norm $l^2(L^2(\Omega)^d)$ (discret in time and continuous in space) is of $\mathcal{O}(\delta t^2+h^{l+1})$, where $l\ge 1$ is the polynomial degree of the velocity approximation. One can read the proof in \cite{guermondCRAS1997} where it is also shown that the splitting error of the projection schemes, based on the incremental pressure correction, is $\mathcal{O}(\delta t^2)$ even if the application of the time derivative of the velocity is $\mathcal{O}(\delta t)$.
	
	The paper is devoted to the proposal of a reduced-order method to the incremental projection scheme. It is organised as follows : in section 2, we give semi-discrete and fully discrete schemes respectively, using the BDF2 scheme in time and the finite element scheme in space.
	In section 3, we mainly introduce POD method and construct a reduced order model about Stokes equation. Meanwhile, we give the corresponding stability analysis and error estimation. In the last section, some numerical examples are calculated to verify the theoretical analysis.

	\section{Second order incremental scheme}
	
	The unsteady Stokes equation   can be expressed as follows
	\begin{subequations}\label{stokes1}
		\begin{align}
			\mathbf{u}_t  - \nu \Delta \mathbf{u} + \nabla p & =f, \quad \text{in}~ \Omega\times]0, T],\\
			\nabla \cdot \mathbf{u} & =0, \quad \text{in}~ \Omega\times]0, T],\\
			\mathbf{u}|_{\partial\Omega} & =0, \quad \text{in}~]0, T],\\
			\mathbf{u}|_{t=0} & =\mathbf{u}^{0}, \quad \text{in}~ \Omega,
		\end{align}
	\end{subequations}
	where $\Omega$ is a bounded domain with a sufficiently smooth boundary $\partial\Omega$, $\mathbf{u}$ denotes the velocity and $p$ stands for the pressure.

	We construct the second-order schemes in time, and use the finite element method in space. 
	
	\subsection{Semi-discrete scheme}
We start by describing the time discretization of \eqref{stokes1} adopted in this work.
	Let $\tau > 0$ be the time step and set $t^k = k \tau$ for $0 \leq k \leq N_t = [T/\tau]$.\\
	Let us take a quick look at the scheme required for this work (see for more details \cite{guermondCRAS1997,guermond2006overview}).\\
	
	\medskip
	We consider the two-step backward differentiation formula (BDF-2): for $n \geq 1$. \\[2mm]
	{\bf{Step} 1}: Calculate $\tilde{\mathbf{u}}^{n+1}$ by
	\begin{subequations}\label{semi_bdf_1}
		\begin{align}
			\frac {3\tilde{\mathbf{u}}^{n+1} - 4\mathbf{u}^n+\mathbf{u}^{n-1}} {2\tau} - \nu \Delta \tilde{\mathbf{u}}^{n+1} + \nabla p^{n} & =f^{n+1},\\
			\tilde{\mathbf{u}}^{n+1}|_{\partial\Omega}&=0.
		\end{align}
	\end{subequations}
	{\bf{Step} 2}: Compute $\phi^{n+1}$ by
	\begin{subequations}\label{semi_bdf_2}
		\begin{align}
			\frac {3\mathbf{u}^{n+1}-3\tilde{\mathbf{u}}^{n+1}} {2\tau} + \nabla \phi^{n+1}&=0,\\
			\nabla\cdot \mathbf{u}^{n+1}&=0,\\
			\mathbf{u}^{n+1}\cdot  \mathbf{n}&=0,
		\end{align}
	\end{subequations}
	where we take $\phi^{n+1}=p^{n+1}-p^n$. From \eqref{semi_bdf_2}, we recover the velocity as  
	\bex
	\mathbf{u}^{n+1} &=  \tilde{\mathbf{u}}^{n+1} - \frac{2}{3}\tau \nabla (p ^{n+1}-p^{n}).
	\eex
	Similarly, at $t=t^n$ time step, we can derive
	\bex
	\mathbf{u}^{n} &=  \tilde{\mathbf{u}}^{n} - \frac{2}{3}\tau \nabla (p ^{n}-p^{n-1}).
	\eex
	Meanwhile, we have the following result at $t^{n-1}$ time step
	\bex
	\mathbf{u}^{n-1} &=  \tilde{\mathbf{u}}^{n-1} -\frac{2}{3} \tau \nabla (p ^{n-1}-p^{n-2}).
	\eex
	Substituting the above equations into \eqref{semi_bdf_1}, and going to the Step 2, the final scheme is reformulated as: \newline 
    For $n \geq 1$, given $\mathbf{\tilde{u}^0}, \mathbf{\tilde{u}^1},p^0,p^1$ and $p^{-1}$ we compute $\mathbf{\tilde{u}^{n+1}}$ by: 
	\begin{subequations}\label{semi_bdf_3}
		\begin{align}
			\frac {3\tilde{\mathbf{u}}^{n+1} - 4\tilde{\mathbf{u}}^n+\tilde{\mathbf{u}}^{n-1}} {2\tau} - \nu \Delta \tilde{\mathbf{u}}^{n+1} + \frac{1}{3}\nabla(7 p^{n}-5 p^{n-1} +p^{n-2} )& =f^{n+1},\\
			\tilde{\mathbf{u}}^{n+1}|_{\partial\Omega}&=0.
		\end{align}
	\end{subequations}
	For the second step, we apply the divergence operator to the equation \eqref{semi_bdf_2}, which can be rewritten as
	\be\label{semi_bdf_5}
	-\nabla\cdot\tilde{\mathbf{u}}^{n+1} + \frac{2\tau}{3} \Delta \phi^{n+1}=0,
	\ee
	where $\phi^{n+1}=p^{n+1}-p^n$.
	So, at the first step we compute $\tilde{\mathbf{u}}^{n+1}$, and then substituting it in \eqref{semi_bdf_5}, we can get $\phi^{n+1}$. Finally, we update $p^{n+1}$ and $\mathbf{u}^{n+1}$ by:
    \begin{align}
    p^{n+1} & =\phi^{n+1}+p^n,  \nonumber \\
    \mathbf{u}^{n+1} & = \mathbf{\tilde{u}}^{n+1} - \frac{2}{3} \tau \nabla \phi^{n+1}.
    \end{align}
    To initialize the BDF2 scheme \eqref{semi_bdf_3}, we compute the first time step using a first-order backward differentiation formula (BDF1), as commonly done in fractional-step methods, see e.g. \cite{haferssas2018efficient}. Given the initial conditions $\tilde{\mathbf{u}}^{0} = \mathbf{u}^0$ and $p^{0}$ (with the convention $p^{-1}=p^{0}$), we first compute $\tilde{\mathbf{u}}^{1}$ from
\begin{subequations}\label{semi_bdf_init}
\begin{align*}
\frac{\tilde{\mathbf{u}}^{1}-\tilde{\mathbf{u}}^{0}}{\tau}
-\nu \Delta \tilde{\mathbf{u}}^{1} + \nabla p^{0} &= f^{1},\\
\tilde{\mathbf{u}}^{1}\big|_{\partial\Omega} &= 0,
\end{align*}
\end{subequations}
and then obtain $\phi^{1}$ from
\begin{equation*}\label{semi_bdf_init_poisson}
-\nabla\cdot\tilde{\mathbf{u}}^{1} + \tau \Delta \phi^{1}=0,
\qquad \phi^{1}=p^{1}-p^{0}.
\end{equation*}
Finally, we update
\begin{equation*}\label{semi_bdf_init_update}
p^{1}=p^{0}+\phi^{1},
\qquad 
\mathbf{u}^{1}=\tilde{\mathbf{u}}^{1}-\tau \nabla \phi^{1}.
\end{equation*}
\begin{remark}\label{rem:start-up}
Since the first time step is computed with a BDF1 initialization, the overall scheme is not globally second-order accurate at the very first iterations. For this reason, when assessing temporal convergence, we compute the order after the start-up phase, discarding the first time steps and measuring errors only on the time interval where the BDF2 scheme is fully initialized.
\end{remark}

	\subsection{Fully discrete scheme}\label{FullyDiscrete}
	Let $\mathcal{T}_h=\{K\}$ be the quasiuniform triangulation of $\Omega$. Define the finite element spaces 
	$$
	\begin{aligned}
        & U_h=\left\{\boldsymbol{v}_h \in \mathbf{H}_0(div,\Omega) \cap C(\bar{\Omega})^2:\left.v_h\right|_K \in \mathcal{P}_2^2(K), \forall K \in \mathcal{T}_h\right\}, \\
		& \tilde{U}_h=\left\{\boldsymbol{v}_h \in \mathbf{H}_0^1(\Omega) \cap C(\bar{\Omega})^2:\left.v_h\right|_K \in \mathcal{P}_2^2(K), \forall K \in \mathcal{T}_h\right\}, \\
		& Q_h=\left\{q_h \in L_0^2(\Omega):\left.q_h\right|_K \in \mathcal{P}_1(K), \forall K \in \mathcal{T}_h\right\},
	\end{aligned}
	$$
	where 
    $$
    \mathbf{H}_0^1(div,\Omega) = \{ v \in \mathbf{L}^2(\Omega) \, : \, \nabla \cdot v \in L^2(\Omega), \, \mathbf{v} \cdot \mathbf{n} = 0 \text{ on } \partial \Omega \},
    $$
    and $\mathcal{P}_1(K)$ is the space of polynomials of degree not exceeding 1.
	The fully discrete scheme is formulated as below: \newline 
 For $n = 0,$ obtain $\tilde{\mathbf{u}}_{h}^{1} \in \tilde{U}_h$ by solving 
		 \begin{equation*}
			 \Big(\frac {\tilde{\mathbf{u}}_{h}^{1} - \tilde{\mathbf{u}}_{h}^0} {\tau},\mathbf{\tilde{v}}_{h}\Big) + \nu (\nabla \tilde{\mathbf{u}}_{h}^{1},\nabla\mathbf{\tilde{v}}_{h}) + (\nabla p^0,\mathbf{\tilde{v}}_{h})
			  =(f^{n+1},\mathbf{\tilde{v}}_{h}),~~ \forall \mathbf{\tilde{v}}_{h}\in \tilde{U}_h,
        \end{equation*}
    and then obtain $\phi_h^1 \in Q_h$ from
        \begin{equation*}
			   (\nabla\cdot \tilde{\mathbf{u}}_{h}^{1}, q_h) + \tau(\nabla \phi_{h}^{1},\nabla q_h) =0, ~~\forall q_{h}\in Q_h.
	   \end{equation*}
       Finally update $p^1_h \in Q_h$ and $\mathbf{u}_h^1 \in U_h$ by solving
       \begin{align}
           (p^1_h,q_h) & = (\phi^1 + p^0,q_h) ~~\forall q_{h}\in Q_h, \\
           (\mathbf{u}_h^{1}, \mathbf{v}_h) & = (\mathbf{\tilde{u}}_h^1, \mathbf{v}_h) - \tau (\nabla \phi^1, \mathbf{v}_h) ~~ \forall \mathbf{v}_{h}\in U_h.
       \end{align}
	 Here we set $p_{h}^{0}=p_{h}^{-1}$, and $\tilde{\mathbf{u}}_{h}^{0}=\mathbf{u}_{h}^{0}$, where $\mathbf{u}^0_h$ is taken as a stable approximation to $\mathbf{u}_0$ in $\mathbf{L}^2-$norm.  \newline
	For $n\geq1$, find $\tilde{\mathbf{u}}_{h}^{n+1}\in \tilde{U}_h$ and $p_{h}^{n+1}\in Q_h$, such that 
	\begin{subequations}\label{fom_bdf_1}
		\begin{align}
			\Big(\frac {3\tilde{\mathbf{u}}_{h}^{n+1} - 4\tilde{\mathbf{u}}_{h}^n+\tilde{\mathbf{u}}_{h}^{n-1}} {2\tau},\tilde{\mathbf{v}}_{h}\Big)
			& + \nu (\nabla \tilde{\mathbf{u}}_{h}^{n+1},\nabla\tilde{\mathbf{v}}_{h}) \nonumber \\
			& + \frac{1}{3}(\nabla (7p_{h}^{n}-5p_{h}^{n-1}+p_{h}^{n-2}),\tilde{\mathbf{v}}_{h})
			=(f^{n+1},\tilde{\mathbf{v}}_{h}),~~ \forall \tilde{\mathbf{v}}_{h}\in \tilde{U}_h,\\
			(\nabla\cdot \tilde{\mathbf{u}}_{h}^{n+1}, q_h) &+ \frac{2 \tau}{3} (\nabla \phi_{h}^{n+1},\nabla q_h)=0, ~~\forall q_{h}\in Q_h,
		\end{align}
	\end{subequations}
then we update $p^{n+1}_h \in Q_h$ and $\mathbf{u}^{n+1}_h \in U_h$ by solving
    \begin{align}
    (p_{h}^{n+1},q_h) & = (\phi_{h}^{n+1}+p_{h}^n,q_h) ~~\forall q_h \in Q_h  \label{fom_bdf_2a}\\
     (\mathbf{u}_h^{n+1}, \mathbf{v}_h) & = (\mathbf{\tilde{u}}_h^{n+1}, \mathbf{v}_h) - \frac23\tau (\nabla \phi^{n+1}, \mathbf{v}_h) ~~ \forall \mathbf{v}_{h}\in U_h. \label{fom_bdf_2b}
    \end{align}
    
	\section{The POD reduced order model}\label{ROMScheme}
	In this section, we introduce the POD method briefly. This method can save computing costs by replacing global basis functions with reduced basis. The detailed presentation can be referred to \cite{kunisch2001galerkin}.
	
	\subsection{The proper orthogonal decomposition}
	Define $\mathcal{V},\tilde{\mathcal{V}}$ and $\mathcal{P}$ as ensembles of snapshots, i.e. $\mathcal{V}=\mathrm{span}\{\mathbf{u}_1, \mathbf{u}_2,\cdots,\mathbf{u}_{N_t+1}\},\tilde{\mathcal{V}}=\mathrm{span}\{\tilde{\mathbf{u}}_1, \tilde{\mathbf{u}}_2,\cdots,\tilde{\mathbf{u}}_{N_t+1}\}$ and  $\mathcal{P}=\mathrm{span}\{p_1, p_2,\cdots,p_{N_t+1}\}$
	where
	$\mathbf{u}_i=\mathbf{u}_h^{n_0 + i-1}(x,y)~~( 1 \leq i \leq N_t+1),$ $\tilde{\mathbf{u}}_i=\tilde{\mathbf{u}}_h^{n_0 + i-1}(x,y)~~( 1 \leq i \leq N_t+1)$ and $p_i=p_h^{n_0+i-1}(x,y)~~( 1 \leq i \leq N_t+1)$, where $n_0 \geq 2$ denotes the first time level at which the full-order model is advanced using the BDF2 scheme. The goal of the POD method is to seek three orthogonal bases $\{\bvarphi_1,\bvarphi_2,\cdots,\bvarphi_{N_u}\}$,$\{\tilde{\bvarphi}_1,\tilde{\bvarphi}_2,\cdots,\tilde{\bvarphi}_{N_{\tilde{u}}}\}$ and $\{\psi_1,\psi_2,\cdots,\psi_{N_p}\}$ to optimize the following minimization problems
	\begin{align}
        &\min_{\{\mathbf{\varphi}_j\}_{j=1}^{N_u}}\frac{1}{N_t+1}\sum_{i=1}^{N_t+1}\|\mathbf{u}_i-\sum_{j=1}^{N_u}(\mathbf{u}_{i}, \bvarphi_{j})_{X} \bvarphi_{j}\|_{X}^2, \label{pod1}\\
		&\min_{\{\mathbf{\tilde{\varphi}}_j\}_{j=1}^{N_{\tilde{u}}}}\frac{1}{N_t+1}\sum_{i=1}^{N_t+1}\|\tilde{\mathbf{u}}_i-\sum_{j=1}^{N_{\tilde{u}}}(\tilde{\mathbf{u}}_{i}, \tilde{\bvarphi}_{j})_{X} \tilde{\bvarphi}_{j}\|_{X}^2, \label{pod2}\\
		&\min_{\{\psi_j\}_{j=1}^{N_p}}\frac{1}{N_t+1}\sum_{i=1}^{N_t+1}\|p_i-\sum_{j=1}^{N_p}(p_{i}, \psi_{j})_{Y} \psi_{j}\|_{Y}^2, \label{pod3}
	\end{align}
	subject to $(\bvarphi_i,\bvarphi_j)_X=\delta_{ij}, 1\leq i,j\leq N_u,$ $(\tilde{\bvarphi}_i,\tilde{\bvarphi}_j)_X=\delta_{ij}, 1\leq i,j\leq N_{\tilde{u}}$ and $(\psi_i,\psi_j)_Y=\delta_{ij}, 1\leq i,j\leq N_p$, where $\delta_{ij}$ is the Kronecker delta, and $X,Y$ are general Hilbert spaces with their respective inner product.
	The solutions of \eqref{pod1}, \eqref{pod2} and \eqref{pod3} are given by solving the eigenvalue problems:
	\begin{align}
        &\mathcal{K}^{\mathbf{u}}\bvarphi_i^{*}= \lambda_i \bvarphi_i^{*}, \quad i=1,2,\cdots, N_t+1, \nonumber\\
		&\tilde{\mathcal{K}}^{\tilde{\mathbf{u}}} \tilde{\bvarphi}_i^{*}= \lambda_i\tilde{\bvarphi}_i^{*}, \quad i=1,2,\cdots, N_t+1, \nonumber\\
		&\mathcal{K}^{p}\tilde{\psi}_i= \lambda_i\tilde{\psi}_i, \quad i=1,2,\cdots, N_t+1,
		\nonumber
	\end{align}
	in which $\mathcal{K}^{\mathbf{u}}=(K^{\mathbf{u}}_{i j}) \in R^{(N_t+1) \times (N_t+1)}$, $\tilde{\mathcal{K}}^{\tilde{\mathbf{u}}}=(\tilde{K}^{\tilde{\mathbf{u}}}_{i j}) \in R^{(N_t+1) \times (N_t+1)}$ and $\mathcal{K}^{p}=(K^{p}_{i j}) \in R^{(N_t+1) \times (N_t+1)}$ denote the correlation matrices, $K^{\mathbf{u}}_{i j}=\frac{1}{N_t+1}( \mathbf{u}_{i},  \mathbf{u}_{j})_X$, $\tilde{K}^{\tilde{\mathbf{u}}}_{i j}=\frac{1}{N_t+1}( \tilde{\mathbf{u}}_{i},  \tilde{\mathbf{u}}_{j})_X$ and  $K^{p}_{i j}=\frac{1}{N_t+1}( p_{i},  p_{j})_Y$,  $\mathbf{\varphi}_i^*,$ $\mathbf{\tilde{\varphi}}_i^*$ and $\tilde{\psi}_i$ are the eigenvector corresponding to the eigenvalue $\lambda_i,$ $\tilde{\lambda_i}$ and $\sigma_i$ respectively, with $\lambda_1 \geq \cdots \geq \lambda_{N_t+1}\geq 0,$ $\tilde{\lambda}_1 \geq \cdots \tilde{\lambda}_{N_t +1 } \geq 0$ and $\sigma_1 \geq \cdots \geq \sigma_{N_t+1} \geq 0.$
	Furthermore, the POD basis functions of rank $N_u,N_{\tilde{u}},N_p \leq{N_t+1}$ are given by the formula
	\begin{align}
        &\bvarphi_i= \frac{1}{\sqrt{(N_t+1)\lambda_i}}\sum_{k=0}^{N_t+1}(\bvarphi_i^*)_k\mathbf{u}_k, \quad i=1,2,\cdots, N_u,
		\nonumber\\
		&\tilde{\bvarphi}_i= \frac{1}{\sqrt{(N_t+1)\tilde{\lambda}_i}}\sum_{k=0}^{N_t+1}(\tilde{\bvarphi}_i^*)_k\tilde{\mathbf{u}}_k, \quad i=1,2,\cdots, N_{\tilde{u}},
		\nonumber\\
		&\psi_i= \frac{1}{\sqrt{(N_t+1)\sigma_i}}\sum_{k=0}^{N_t+1}(\tilde{\psi}_i)_kp_k, \quad i=1,2,\cdots, N_p,
		\nonumber
	\end{align}
	where $(\bvarphi_i^*)_k$, $(\tilde{\bvarphi}_i^*)_k$ and $(\tilde{\psi}_i)_k$ denote the $k$-th component of $\bvarphi_i^*$, $\tilde{\bvarphi}_i^*$ and $\tilde{\psi}_i$, respectively.
	Then the errors of POD method are given by
	\begin{align}\label{projerror}
		&\frac{1}{N_t+1} \sum_{i=1}^{N_t+1}\|\mathbf{u}_{i}-\sum_{j=1}^{N_u}(\mathbf{u}_{i}, \bvarphi_{j})_{X} \bvarphi_{j}\|_{X}^{2}=\sum_{j=N_u+1}^{{M_u}} \lambda_{j},
		 \\
        &\frac{1}{N_t+1} \sum_{i=1}^{N_t+1}\|\tilde{\mathbf{u}}_{i}-\sum_{j=1}^{N_{\tilde{u}}}(\tilde{\mathbf{u}}_{i}, \tilde{\bvarphi}_{j})_{X} \tilde{\bvarphi}_{j}\|_{X}^{2}=\sum_{j=N_{\tilde{u}}+1}^{{M_{\tilde{u}}}} \tilde{\lambda}_{j},
		 \\
		&\frac{1}{N_t+1} \sum_{i=1}^{N_t+1}\|p_{i}-\sum_{j=1}^{N_p}(p_{i}, \psi_{j})_{Y} \psi_{j}\|_{Y}^{2}=\sum_{j=N_p+1}^{{M_p}} \sigma_{j},
	\end{align}
	where $M_u,M_{\tilde{u}}$ and $M_p$ are the rank of the snapshots $\mathcal{V}, \mathcal{\tilde{V}}$ and $\mathcal{P}$ respectively. We are now in position to construct the reduced model for unsteady Stokes by using the POD basis.
	
	\subsection{Formulating projection ROM}
Let
	$$
	\begin{aligned}
        & U_r=\span\{\bvarphi_1,\bvarphi_2,\cdots,\bvarphi_{N_u}\}, \\
		& \tilde{U}_r=\span\{\tilde{\bvarphi}_1,\tilde{\bvarphi}_2,\cdots,\tilde{\bvarphi}_{N_{\tilde{u}}}\}, \\
		& Q_r=\span\{\psi_1,\psi_2,\cdots,\psi_{N_p}\}.
	\end{aligned}
	$$
    The initial solutions required by the BDF2 scheme are obtained by projection of the corresponding full-order solutions onto the reduced spaces. In particular, for $k=0,1,2$, we define
\[
\qquad 
\tilde{\mathbf{u}}_r^k = \sum_{i=1}^{N_{\tilde{u}}} (\tilde{\mathbf{u}}_k,\tilde{\bm{\varphi}}_i)_X\tilde{\bm{\varphi}}_i
\qquad 
p_r^k = \sum_{i=1}^{N_p} (p_k,\psi_i)_Y\psi_i,
\]  
  Then, for $n \geq 2:$
	\begin{subequations}\label{rom_bdf_1}
		\begin{align}
			\frac {3\tilde{\mathbf{u}}_{r}^{n+1} - 4\tilde{\mathbf{u}}_{r}^n+\tilde{\mathbf{u}}_{r}^{n-1}} {2\tau}
			- \nu \Delta  \tilde{\mathbf{u}}_{r}^{n+1}
			+ \frac{1}{3}\nabla (7p_{r}^{n}-5p_{r}^{n-1}+p_{r}^{n-2})
			& =f^{n+1},~~\text{in}~\Omega,\\
			-\nabla\cdot \tilde{\mathbf{u}}_{r}^{n+1} + \frac{2 \tau}{3} \Delta p_{r}^{n+1}&=  \frac{2 \tau}{3} \Delta p_{r}^{n},~~\text{in}~\Omega, \\
            \mathbf{u}_r^{n+1}  &= \mathbf{\tilde{u}}_r^{n+1} - \frac23 \nabla(p_r^{n+1} - p_r^n), ~~\text{in} ~ \Omega. 
		\end{align}
	\end{subequations}
	Below we give the matrix calculation form. We first expand $\mathbf{u}_r, \tilde{\mathbf{u}}_{r}, p_{r}$ by POD basis $\mathbf{\varphi}_{i}(\mathbf{x}), \mathbf{\tilde{\varphi}}_{i}(\mathbf{x}), \psi_{i}(\mathbf{x})$ respectively
    $$
	\begin{aligned}
         & \mathbf{u}_{r}(\mathbf{x},t)=\sum_{i=1}^{N_u}a_{i}(t)\bvarphi_{i}(\mathbf{x}),\\
		& \tilde{\mathbf{u}}_{r}(\mathbf{x},t)=\sum_{i=1}^{N_{\tilde{u}}}\tilde{a}_{i}(t)\tilde{\bvarphi}_{i}(\mathbf{x}),\\
		& p_{r}(\mathbf{x},t)=\sum_{i=1}^{N_p}b_{i}(t)\psi_{i}(\mathbf{x}).
	\end{aligned}
    $$
	Substituting the above equations to \eqref{rom_bdf_1} under the Galerkin framework, we obtain \newline 
    For $n \geq 2:$
	\begin{subequations}\label{rom_bdf_2}
		\begin{align}
			\Big(\frac {3\tilde{\mathbf{u}}_{r}^{n+1} - 4\tilde{\mathbf{u}}_{r}^n+\tilde{\mathbf{u}}_{r}^{n-1}} {2\tau},  \tilde{\mathbf{v}}_{r}\Big)
			&+ \nu (\nabla {\tilde{\mathbf{u}}}_{r}^{n+1},\tilde{\mathbf{v}}_{r})\nonumber\\
			&- \frac{1}{3}( 7p_{r}^{n}-5p_{r}^{n-1}+p_{r}^{n-2},\nabla\cdot \tilde{\mathbf{v}}_{r})
			=(f^{n+1},\tilde{\mathbf{v}}_{r}),~~\forall \tilde{\mathbf{v}}_{r}\in \tilde{U}_r, \label{rom_bdf_2a}\\
			(\nabla \cdot \tilde{\mathbf{u}}_{r}^{n+1},q_{r})
			&+ \frac{2 \tau}{3} (\nabla p_{r}^{n+1},\nabla q_{r})=\frac{2 \tau}{3} (\nabla p_{r}^{n},\nabla q_{r}),
			~~\forall q_{r}\in Q_r, \label{rom_bdf_2b} \\
            &(\mathbf{u}_r^{n+1}, \mathbf{v}_r) = (\mathbf{\tilde{u}}_r^{n+1}, \mathbf{v}_r) ~~\forall \mathbf{v}_r \in U_r. \label{rom_bdf_2c}
		\end{align}
	\end{subequations}
	Here we choose the test functions $\mathbf{v}_{r}=\bvarphi_{i}(\mathbf{x}), i=1,\cdots, N_u$, $\tilde{\mathbf{v}}_{r}=\tilde{\bvarphi}_{i}(\mathbf{x}), i=1,\cdots, N_{\tilde{u}},$ $q_r=\psi_{j}(\mathbf{x}), j=1,\cdots, N_p$. Thus, the following reduced system is obtained: \newline 
    For $n \geq 2:$
	\begin{subequations}\label{rom_bdf_3}
		\begin{align}
			\big(\frac{3}{2\tau} \tilde{\mathbf{M}}_{r}+\mathbf{S}^{\tilde{u}}_{r}\big)\tilde{\mathbf{a}}^{n+1}
			& = \frac{1}{2\tau}\tilde{\mathbf{M}}_{r}(4\tilde{\mathbf{a}}^{n}-\tilde{\mathbf{a}}^{n-1})
			+\frac{1}{3}\mathbf{P}(7\mathbf{b}^{n}-5\mathbf{b}^{n-1}+\mathbf{b}^{n-2})+\mathbf{F}_{r},  \label{rom3a}\\
			\mathbf{S}^{p}_r\mathbf{b}^{n+1}
			&=\mathbf{S}^{p}_r\mathbf{b}^{n} -\frac{3}{2\tau}\mathbf{P}^{T}\tilde{\mathbf{a}}^{n+1},    \label{rom3b} \\
            \mathbf{M}_r \mathbf{a}^{n+1} 
            & = \hat{\mathbf{M}} \tilde{\mathbf{a}}^{n+1}. 
		\end{align}
	\end{subequations}
	The matrices $\tilde{\mathbf{M}}_{r}\in R^{N_{\tilde{u}}\times N_{\tilde{u}}}$, $\mathbf{S}^{\tilde{u}}_{r}\in R^{N_{\tilde{u}}\times N_{\tilde{u}}}$, $\mathbf{F}_{r}\in R^{N_{\tilde{u}}\times1}$, $\mathbf{P}\in R^{N_{\tilde{u}}\times N_p}$, $\mathbf{S}^{p}_{r}\in R^{N_p\times N_p}$, $\mathbf{M}_r \in R^{N_{u} \times N_u},$ $\hat{\mathbf{M}}_r \in R^{N_u \times N_{\tilde{u}}}$ appearing in the system are defined as follows:
    $$
	\begin{aligned}
		&  (\tilde{\mathbf{M}}_r)_{ij}= 
		(\tilde{\bvarphi}_{i},\tilde{\bvarphi}_{j}),\quad
		(\mathbf{S}^{\tilde{u}}_{r})_{ij}=(\nabla \tilde{\bvarphi}_{i},\nabla \tilde{\bvarphi}_{j}), \quad
		(\mathbf{F})_{ij}=(f,\tilde{\bvarphi}_{j}), \\
		&  (\mathbf{P})_{ij}=(\psi_{i},\nabla\cdot \tilde{\bvarphi}_{j}), \quad
		(\mathbf{S}^{p}_{r})_{ij}=(\nabla\psi_{i},\nabla\psi_{j}), \quad (\mathbf{M}_r)_{ij} = (\mathbf{\varphi}_i, \mathbf{\varphi}_j), \quad (\hat{\mathbf{M}}_r)_{ij} = (\mathbf{\varphi}_i, \tilde{\mathbf{\varphi}}_j). 
	\end{aligned}
    $$

	\subsection{Stability and error estimate of pressure projection ROM}
	Firstly, we prove that the algorithm \eqref{rom_bdf_2} is stable.
	
	\begin{theorem}
		The solution $\{(\mathbf{u}_r^n,\tilde{\mathbf{u}}_{r}^{n},p_{r}^{n})\}$ to \eqref{rom_bdf_2} satisfies the following estimates:
		for $n>2$
        
        \begin{align}\label{rom_st_0}
           \| \mathbf{u}_r^n \| \le c(\|\tilde{\mathbf{u}}_{r}^{1}\|_{0}^2+\|\tilde{\mathbf{u}}_{r}^{2}\|_{0}^2
			+\tau^2\sum_{i=0}^2\|\nabla p_{r}^{i}\|_{0}^2 + \nu^{-1} \sum_{k=2}^{N_t} \tau \| f^{k}\|_{\bm{H}^{-1}}^2).
        \end{align}
        
		\begin{align}\label{rom_st_1}
			\|\tilde{\mathbf{u}}_{r}^{n}\|_{0}^2+\tau^2\|\nabla p_{r}^{n}\|_{0}^2
			+\tau^2\|\nabla (p_{r}^{n-1}-p_{r}^{n-2})\|_{0}^2
			+\sum_{k=2}^{{N_t}}(\nu\tau \|\nabla\tilde{\mathbf{u}}_{r}^{k}\|_{0}^2
			+\tau^2\|\nabla (p_{r}^{k-1}-p_{r}^{k-2})\|_{0}^2) \nonumber \\
			\le c(\|\tilde{\mathbf{u}}_{r}^{1}\|_{0}^2+\|\tilde{\mathbf{u}}_{r}^{2}\|_{0}^2
			+\tau^2\sum_{i=0}^2\|\nabla p_{r}^{i}\|_{0}^2 + \nu^{-1} \sum_{k=2}^{N_t} \tau \| f^{k}\|_{\bm{H}^{-1}}^2).
		\end{align}
	\end{theorem}
	\begin{proof}
First, we estimate the velocity $\mathbf{u}_r^{n+1},$ by taking $\mathbf{v}_r =  \mathbf{u}_r^{n+1}$ in \eqref{rom_bdf_2c} and applying Cauchy-Schwarz and Young's inesqualities: 
    \begin{equation}\label{eq:rom_pf_0}
        \| \mathbf{u}_r^{n+1} \|_0^2 \leq \| \mathbf{\tilde{u}}_r^{n+1} \|_0^2.
    \end{equation}
    Therefore, estimate \eqref{rom_st_0} follows from \eqref{eq:rom_pf_0} by proving \eqref{rom_st_1} in the following step. 
		To prove \eqref{rom_st_1}, we follow the idea of Theorem 5.1 in \cite{guermond2011error}. For ease of reading, we briefly describe it.
		Taking $\mathbf{v}_{r}=4\tau\tilde{\mathbf{u}}_{r}^{n+1}$ in \eqref{rom_bdf_2} and using the identity
		\be\label{rom_st_pf_5}
		2a(3a-4b+c)=3a^2-4b^2+c^2+2(a-b)^2-2(b-c)^2+(a-2b+c)^2,
		\nonumber
		\ee
		we have
		\begin{equation}\label{rom_st_pf_6}
			\begin{split}
				&3\|\tilde{\mathbf{u}}_{r}^{n+1}\|_{0}^2-4\|\tilde{\mathbf{u}}_{r}^{n}\|_{0}^2+\|\tilde{\mathbf{u}}_{r}^{n-1}\|_{0}^2
				+2\|\tilde{\mathbf{u}}_{r}^{n+1}-\tilde{\mathbf{u}}_{r}^{n}\|_{0}^2
				-2\|\tilde{\mathbf{u}}_{r}^{n}-\tilde{\mathbf{u}}_{r}^{n-1}\|_{0}^2\\
				&+\|\tilde{\mathbf{u}}_{r}^{n+1}-2\tilde{\mathbf{u}}_{r}^{n}+\tilde{\mathbf{u}}_{r}^{n-1}\|_{0}^2
				+4\tau \nu\|\nabla\tilde{\mathbf{u}}_{r}^{n+1}\|_{0}^2
				+4\tau(\nabla p_{r}^{n+1},\tilde{\mathbf{u}}_{r}^{n+1})\\
				&-4\tau(\nabla (p_{r}^{n+1}-2 p_{r}^{n}+p_{r}^{n-1}),\tilde{\mathbf{u}}_{r}^{n+1})
				+\frac{4\tau}{3}(\nabla (p_{r}^{n}-2p_{r}^{n-1}+p_{r}^{n-2}),\tilde{\mathbf{u}}_{r}^{n+1})
				=4 \tau \langle f^{n+1}, \tilde{\mathbf{u}}_r^{n+1} \rangle.
			\end{split}
		\end{equation}
		Choosing $q_r=p_{r}^{n+1}$, $p_{r}^{n+1}-2 p_{r}^{n}+p_{r}^{n-1}$ and $p_{r}^{n}-2p_{r}^{n-1}+p_{r}^{n-2}$ in \eqref{rom_bdf_2b} respectively, the results list as below
		\begin{equation}\label{rom_st_pf_7}
			\begin{split}
				(\tilde{\mathbf{u}}_{r}^{n+1},\nabla p_{r}^{n+1})
				&= \frac{2 \tau}{3} (\nabla (p_{r}^{n+1}-p_{r}^{n}),\nabla p_{r}^{n+1})\\
				&= \frac{\tau}{3} (\|\nabla p_{r}^{n+1}\|_{0}^2-\|\nabla p_{r}^{n}\|_{0}^2+\|\nabla (p_{r}^{n+1}-p_{r}^{n})\|_{0}^2),
			\end{split}
		\end{equation}
		\begin{equation}\label{rom_st_pf_8}
			\begin{split}
				&(\tilde{\mathbf{u}}_{r}^{n+1},\nabla (p_{r}^{n+1}-2 p_{r}^{n}+p_{r}^{n-1}))
				= \frac{2 \tau}{3} (\nabla (p_{r}^{n+1}-p_{r}^{n}),\nabla (p_{r}^{n+1}-2 p_{r}^{n}+p_{r}^{n-1}))\\
				&= \frac{\tau}{3} (\|\nabla (p_{r}^{n+1}-p_{r}^{n})\|_{0}^2-\|\nabla (p_{r}^{n}-p_{r}^{n-1})\|_{0}^2
				+\|\nabla (p_{r}^{n+1}-2p_{r}^{n}+p_{r}^{n-1})\|_{0}^2),
			\end{split}
		\end{equation}
		\begin{equation}\label{rom_st_pf_9}
			\begin{split}
				(\tilde{\mathbf{u}}_{r}^{n+1},\nabla (p_{r}^{n}-2p_{r}^{n-1}+p_{r}^{n-2}))
				= \frac{2 \tau}{3} (\nabla (p_{r}^{n+1}-p_{r}^{n}),\nabla (p_{r}^{n}-2p_{r}^{n-1}+p_{r}^{n-2})).
			\end{split}
		\end{equation}
		Putting \eqref{rom_st_pf_7}-\eqref{rom_st_pf_9} into \eqref{rom_st_pf_6} we get
		\begin{equation}\label{rom_st_pf_10}
			\begin{split}
				&3\|\tilde{\mathbf{u}}_{r}^{n+1}\|_{0}^2-4\|\tilde{\mathbf{u}}_{r}^{n}\|_{0}^2+\|\tilde{\mathbf{u}}_{r}^{n-1}\|_{0}^2
				+2\|\tilde{\mathbf{u}}_{r}^{n+1}-\tilde{\mathbf{u}}_{r}^{n}\|_{0}^2
				-2\|\tilde{\mathbf{u}}_{r}^{n}-\tilde{\mathbf{u}}_{r}^{n-1}\|_{0}^2\\
				&+\|\tilde{\mathbf{u}}_{r}^{n+1}-2\tilde{\mathbf{u}}_{r}^{n}+\tilde{\mathbf{u}}_{r}^{n-1}\|_{0}^2
				+4\tau\|\nabla\tilde{\mathbf{u}}_{r}^{n+1}\|_{0}^2
				+\frac{4\tau^2}{3} (\|\nabla p_{r}^{n+1}\|_{0}^2-\|\nabla p_{r}^{n}\|_{0}^2\\
				&+\|\nabla (p_{r}^{n}-p_{r}^{n-1})\|_{0}^2
				-\|\nabla (p_{r}^{n+1}-2p_{r}^{n}+p_{r}^{n-1})\|_{0}^2)\\
				&+\frac{8\tau^2}{9}(\nabla (p_{r}^{n+1}-p_{r}^{n}),\nabla (p_{r}^{n}-2p_{r}^{n-1}+p_{r}^{n-2}))
				= 4 \tau \langle f^{n+1}, \tilde{\mathbf{u}}_r^{n+1} \rangle.
			\end{split}
		\end{equation}
		Now we deal with the last term on the right hand side of \eqref{rom_st_pf_10}. Taking the difference at time $t^{n+1}$ and $t^{n}$ of \eqref{rom_bdf_2b}, we derive
		\begin{equation}\label{rom_st_pf_11}
			\begin{split}
				(\tilde{\mathbf{u}}_{r}^{n+1}-\tilde{\mathbf{u}}_{r}^{n},\nabla q_{r}^{n})
				= \frac{2 \tau}{3} (\nabla (p_{r}^{n+1}-2p_{r}^{n}+p_{r}^{n-1}),\nabla q_{r}).
			\end{split}
		\end{equation}
		Taking $q_{r}=\frac{2 \tau}{3}(p_{r}^{n+1}-2p_{r}^{n}+p_{r}^{n-1})$ in the above equation, we deduce that
		\begin{equation}\label{rom_st_pf_12}
			\begin{split}
				(\tilde{\mathbf{u}}_{r}^{n+1}-\tilde{\mathbf{u}}_{r}^{n},\nabla (p_{r}^{n+1}-2p_{r}^{n}+p_{r}^{n-1}))
				= \frac{4 \tau^2}{9} \|\nabla (p_{r}^{n+1}-2p_{r}^{n}+p_{r}^{n-1})\|_{0}^2.
			\end{split}
		\end{equation}
		Using the equation $(a-b)^2=a^2-2ab+b^2$, the above equation can be rewritten as
		\begin{equation}\label{rom_st_pf_13}
			\begin{split}
				\|\tilde{\mathbf{u}}_{r}^{n+1}-\tilde{\mathbf{u}}_{r}^{n}\|_{0}^2
				=\|\tilde{\mathbf{u}}_{r}^{n+1}-\tilde{\mathbf{u}}_{r}^{n}-\frac{2\tau}{3} \nabla (p_{r}^{n+1}-2p_{r}^{n}+p_{r}^{n-1})\|_{0}^2
				+ \frac{4 \tau^2}{9} \|\nabla (p_{r}^{n+1}-2p_{r}^{n}+p_{r}^{n-1})\|_{0}^2.
			\end{split}
		\end{equation}
		Substituting \eqref{rom_st_pf_13} into \eqref{rom_st_pf_10} we obtain
		\begin{equation}\label{rom_st_pf_14}
			\begin{split}
				&3\|\tilde{\mathbf{u}}_{r}^{n+1}\|_{0}^2-4\|\tilde{\mathbf{u}}_{r}^{n}\|_{0}^2+\|\tilde{\mathbf{u}}_{r}^{n-1}\|_{0}^2
				+2\|\tilde{\mathbf{u}}_{r}^{n+1}-\tilde{\mathbf{u}}_{r}^{n}
				-\frac{2\tau}{3} \nabla (p_{r}^{n+1}-2p_{r}^{n}+p_{r}^{n-1})\|_{0}^2\\
				&-2\|\tilde{\mathbf{u}}_{r}^{n}-\tilde{\mathbf{u}}_{r}^{n-1}
				-\frac{2\tau}{3} \nabla (p_{r}^{n}-2p_{r}^{n-1}+p_{r}^{n-2})\|_{0}^2
				+\|\tilde{\mathbf{u}}_{r}^{n+1}-2\tilde{\mathbf{u}}_{r}^{n}+\tilde{\mathbf{u}}_{r}^{n-1}\|_{0}^2\\
				&+4\tau \nu\|\nabla\tilde{\mathbf{u}}_{r}^{n+1}\|_{0}^2
				+\frac{4\tau^2}{3} (\|\nabla p_{r}^{n+1}\|_{0}^2-\|\nabla p_{r}^{n}\|_{0}^2
				+\|\nabla (p_{r}^{n}-p_{r}^{n-1})\|_{0}^2)\\
				&- \frac{4\tau^2}{9}\|\nabla (p_{r}^{n+1}-2p_{r}^{n}+p_{r}^{n-1})\|_{0}^2
				-\frac{8 \tau^2}{9} \|\nabla (p_{r}^{n}-2p_{r}^{n-1}+p_{r}^{n-2})\|_{0}^2\\
				&+\frac{8\tau^2}{9}(\nabla (p_{r}^{n+1}-p_{r}^{n}),\nabla (p_{r}^{n}-2p_{r}^{n-1}+p_{r}^{n-2}))
				=4 \langle f^{n+1}, \tilde{\mathbf{u}}_r^{n+1} \rangle.
			\end{split}
		\end{equation}
		Now we consider the last three terms on the right hand side of \eqref{rom_st_pf_14}.
		\begin{equation}\label{rom_st_pf_15}
			\begin{split}
				&- \frac{4\tau^2}{9}\|\nabla (p_{r}^{n+1}-2p_{r}^{n}+p_{r}^{n-1})\|_{0}^2
				-\frac{8 \tau^2}{9} \|\nabla (p_{r}^{n}-2p_{r}^{n-1}+p_{r}^{n-2})\|_{0}^2\\
				&+\frac{8\tau^2}{9}(\nabla (p_{r}^{n+1}-p_{r}^{n}),\nabla (p_{r}^{n}-2p_{r}^{n-1}+p_{r}^{n-2}))\\
				=&- \frac{4\tau^2}{9}\|\nabla (p_{r}^{n+1}-3p_{r}^{n}+3p_{r}^{n-1}-p_{r}^{n-2})\|_{0}^2
				-\frac{4\tau^2}{9}(\nabla (p_{r}^{n}-2p_{r}^{n-1}+p_{r}^{n-2}),\nabla (-p_{r}^{n}+p_{r}^{n-2}))\\
				=&- \frac{4\tau^2}{9}\|\nabla (p_{r}^{n+1}-3p_{r}^{n}+3p_{r}^{n-1}-p_{r}^{n-2})\|_{0}^2
				+\frac{4\tau^2}{9}(\nabla (p_{r}^{n}-2p_{r}^{n-1}+p_{r}^{n-2}),\nabla ((p_{r}^{n}-p_{r}^{n-1})+(p_{r}^{n-1}-p_{r}^{n-2}))) . \nonumber
			\end{split}
		\end{equation}
		Combining different time steps of \eqref{rom_bdf_2b}, we have
		\be\label{rom_st_pf_16}
		\frac{4\tau^2}{9}\|\nabla (p_{r}^{n+1}-3p_{r}^{n}+3p_{r}^{n-1}-p_{r}^{n-2})\|_{0}^2
		\le \|\tilde{\mathbf{u}}_{r}^{n+1}-2\tilde{\mathbf{u}}_{r}^{n}+\tilde{\mathbf{u}}_{r}^{n-1}\|_{0}^2.
		\ee
		Then we have
		\begin{equation}\label{rom_st_pf_17}
			\begin{split}
				&- \frac{4\tau^2}{9}\|\nabla (p_{r}^{n+1}-2p_{r}^{n}+p_{r}^{n-1})\|_{0}^2
				-\frac{8 \tau^2}{9} \|\nabla (p_{r}^{n}-2p_{r}^{n-1}+p_{r}^{n-2})\|_{0}^2\\
				&+\frac{8\tau^2}{9}(\nabla (p_{r}^{n+1}-p_{r}^{n}),\nabla (p_{r}^{n}-2p_{r}^{n-1}+p_{r}^{n-2}))\\
				\ge & -\|\tilde{\mathbf{u}}_{r}^{n+1}-2\tilde{\mathbf{u}}_{r}^{n}+\tilde{\mathbf{u}}_{r}^{n-1}\|_{0}^2
				+ \frac{4\tau^2}{9}(\|\nabla (p_{r}^{n}-p_{r}^{n-1})\|_{0}^2-\|\nabla (p_{r}^{n-1}-p_{r}^{n-2})\|_{0}^2).
			\end{split}
		\end{equation}
		Thus, gathering all results, we can deduce the following estimate
		\begin{equation}\label{rom_st_pf_18}
			\begin{split}
				&3\|\tilde{\mathbf{u}}_{r}^{n+1}\|_{0}^2-4\|\tilde{\mathbf{u}}_{r}^{n}\|_{0}^2+\|\tilde{\mathbf{u}}_{r}^{n-1}\|_{0}^2
				+2\|\tilde{\mathbf{u}}_{r}^{n+1}-\tilde{\mathbf{u}}_{r}^{n}
				-\frac{2\tau}{3} \nabla (p_{r}^{n+1}-2p_{r}^{n}+p_{r}^{n-1})\|_{0}^2\\
				&-2\|\tilde{\mathbf{u}}_{r}^{n}-\tilde{\mathbf{u}}_{r}^{n-1}
				-\frac{2\tau}{3} \nabla (p_{r}^{n}-2p_{r}^{n-1}+p_{r}^{n-2})\|_{0}^2 +2\tau \nu\|\nabla\tilde{\mathbf{u}}_{r}^{n+1}\|_{0}^2\\
				&
				+\frac{4\tau^2}{3} (\|\nabla p_{r}^{n+1}\|_{0}^2-\|\nabla p_{r}^{n}\|_{0}^2
				+\|\nabla (p_{r}^{n}-p_{r}^{n-1})\|_{0}^2)\\
				&+\frac{4\tau^2}{9}(\|\nabla (p_{r}^{n}-p_{r}^{n-1})\|_{0}^2-\|\nabla (p_{r}^{n-1}-p_{r}^{n-2})\|_{0}^2) - 2 \tau \nu^{-1} \| f^{n+1} \|_{\bf{H}^{-1}}^2
				\le 0.
			\end{split}
		\end{equation}
		Here we denote 
		$$
		\begin{aligned}
			&a^{k+1}=\|\tilde{\mathbf{u}}_{r}^{n+1}\|_{0}^2, \\
			&b^{k+1}=2\tau\|\nabla\tilde{\mathbf{u}}_{r}^{n+1}\|_{0}^2 + \frac{4\tau^2}{3}\|\nabla (p_{r}^{n}-p_{r}^{n-1})\|_{0}^2 - 2\tau \nu^{-1} \| f^{n+1} \|_{H^{-1}}^2, \\
			&d^{k+1}=2\|\tilde{\mathbf{u}}_{r}^{n+1}-\tilde{\mathbf{u}}_{r}^{n}
			-\frac{2\tau}{3} \nabla (p_{r}^{n+1}-2p_{r}^{n}+p_{r}^{n-1})\|_{0}^2+\frac{4\tau^2}{3} \|\nabla p_{r}^{n+1}\|_{0}^2+\frac{4\tau^2}{9}\|\nabla (p_{r}^{n}-p_{r}^{n-1})\|_{0}^2.
		\end{aligned}
		$$
		Using the inequality of three term recursion that appears in \cite{guermond2011error} (Corollary 5.1),
		let $\{a^k\}_{k\ge1}$ be the solution to the three term recursion inequality
		\be\label{rom_st_pf_19}
		3a^{k+1}-4a^{k}+a^{k-1}\le g^{k+1},
		\ee
		with the initial condition $a^{0}$ and $a^{1}$, where $g^{k+1}=-(b^{k+1}+d^{k+1}-d^{k})$. 
		Then the following estimate holds
		\be\label{rom_st_pf_20}
		a^{\nu}+\frac{1}{3}d^{\nu}+\frac{1}{3}\sum_{s=2}^{\nu}b^{s}\le c(a^1+a^2+d^2),\quad \nu\geq 3
		\ee
		where $c\in \mathbb{R}$.
		Finally, we omit some positive terms, then conclude the result \eqref{rom_st_1}.
	\end{proof}

	We now claim the error analysis as follows. 
	Define $\Pi_{\mathbf{v}}: W_u(\Omega)\rightarrow U_{r},$ $\Pi_{\mathbf{\tilde{v}}}: W_{\tilde{u}}(\Omega)\rightarrow \tilde{U}_{r}$ and
	$\Pi_{q}: W_p(\Omega)\rightarrow Q_{r}$ such that
	\bex
    ((\mathbf{u}-\Pi_{\mathbf{v}}\mathbf{u}),\mathbf{v})_{W_u}=0,
    \quad \mathbf{v}\in U_{r}, \\
	((\mathbf{\tilde{u}}-\Pi_{\mathbf{\tilde{v}}}\mathbf{\tilde{u}}),\mathbf{\tilde{v}})_{W_{\tilde{u}}}=0,
    \quad \mathbf{\tilde{v}}\in \tilde{U}_{r}, \\
	((p-\Pi_{q}p), q)_{W_p}=0,
    \quad q\in Q_{r}.
	\eex
	From the estimate of POD error, we have
	\begin{align}
        &\frac{1}{N_t+1} \sum_{i=1}^{N_t+1}\|\mathbf{u}_{i}-\Pi_{\mathbf{v}}\mathbf{u}_{i}\|_{W_u}^{2}=\sum_{j=N_u+1}^{{M_u}} \lambda_{j},
		\nonumber \\
		&\frac{1}{N_t+1} \sum_{i=1}^{N_t+1}\|\mathbf{\tilde{u}}_{i}-\Pi_{\mathbf{v}}\mathbf{\tilde{u}}_{i}\|_{W_u}^{2}=\sum_{j=N_{\tilde{u}}+1}^{{M_{\tilde{u}}}} \tilde{\lambda}_{j},
		\nonumber \\
		&\frac{1}{N_t+1} \sum_{i=1}^{N_t+1}\|p_{i}-\Pi_{q}p_{i}\|_{W_p}^{2}=\sum_{j=N_p+1}^{{M_p}} \sigma_{j}.
		\nonumber
	\end{align}
    
	Note that we define a general inner product to build both reduced spaces, velocity and pressure, respectively.

    Before to state the error estimation theorem, we need to define the following results:

    \begin{definition}
Let $X$ be a Hilbert space and $Y,Z$ two finite-dimensional subspaces of $X$ with intersection reduced to the zero function. 
The pair of finite-dimensional spaces $(Y,Z)$ is called to satisfy the saturation property if there exists a positive constant C such that:
\begin{equation}
    \|y\|_X + \|z\|_X \leq C \|y+z\|_X \quad \forall y \in Y, z \in Z. 
\end{equation}
\end{definition}
The saturation property can be viewed as an inverse triangular inequality. 

\begin{lemma}\label{lemaSaturation}(see \cite{samu}, Lemma 5.3).
The saturation property is equivalent to the existence of a constant $\alpha < 1$ such that:
\begin{equation}
    |(y,z)_X| \leq \alpha \|y\|_X \|z\|_X \quad\forall y \in Y, z \in Z. 
\end{equation}
\end{lemma}
We will denote by $\alpha$ the saturation constant. Then, we can interpret the saturation property in the sense that the angle between spaces $Y$ and $Z,$ defined by:

\begin{equation}\label{alphacomputation}
    \theta = \arccos \left( \sup_{y \in Y \setminus \{0\}, z \in Z \setminus\{0\}} \dfrac{(y,z)_X}{\|y\|_X \|z\|_X} \right),
\end{equation}
is uniformly bounded from below by a positive angle, and $\alpha = \cos(\theta).$\\

\begin{remark}\label{RemarkSatConst}
In the following error estimation theorem, we will use Lemma \ref{lemaSaturation} to bound some of the pressure error terms. In particular, the spaces involved in the application of the saturation lemma are $ Y = \tilde{U}_r$ and $Z = \text{span}\{\nabla \psi_1, \ldots, \nabla \psi_{N_p}\}$  with $X = L^2$ the Hilbert space considered. 
\end{remark}
\begin{remark}\label{RemarkAlpha}
The value of $\alpha$ can be directly estimated from equation~\eqref{alphacomputation}. Moreover, Theorem~\ref{th:errorestimation} imposes a restriction on $\alpha$. Numerical studies already available in the literature (see Section 6 of \cite{samu}) show that this restriction is typically satisfied: $\alpha$ takes very small values, of the order of $10^{-5}$, with a few number of POD basis and it increases as the number of POD basis functions is enlarged. Therefore, when only a few POD modes are retained, the condition required in Theorem \ref{th:errorestimation} is guaranteed.
\end{remark}

	\begin{theorem}\label{th:errorestimation}
		Let $\{(\tilde{\mathbf{u}}_{r}^{n},p_{r}^{n})\}$ and $\{(\tilde{\mathbf{u}}_{h}^{n},p_{h}^{n})\}$ be the solution of the problems \eqref{rom_bdf_2} and \eqref{fom_bdf_1}, respectively. If the saturation constant $\alpha$ of Lemma \ref{lemaSaturation} verifies $\alpha < 3/8,$ then the following estimates hold:
		for $n\ge2$

        \begin{align}\label{rom_ea_th_0}
            \sum_{n=2}^{N_t} \tau \| \mathbf{u}_h^{n+1} - \mathbf{u}_r^{n+1} \|_0^2 \leq C(\alpha) e^{T \frac{7 \tau + 4\alpha}{2}} \Big{(} \sum_{i=N_{\tilde{u}} +1}^{M_{\tilde{u}}} \tilde{\lambda}_i( \nu \tau^{-2} \| \bm{\tilde{\varphi}}_i\|_0^2 + \|\nabla \bm{\tilde{\varphi}}_i\|_0^2)  \nonumber \\  
            + \sum_{i=N_p+1}^{M_p} \sigma_i( \nu \| \psi_i \|_0^2 + \| \nabla \psi_i \|_0^2 )\Big{)} + 5\sum_{i=N_u+1}^{M_u} \lambda_i \| \bm{\varphi}_i \|_0^2
        \end{align}
        
		\begin{align}\label{rom_ea_th}
			&\sum_{n=2}^{N_t} \tau \|\tilde{\mathbf{u}}_{r}^{n+1}-\tilde{\mathbf{u}}_{h}^{n+1}\|_0^2 + 2\nu \sum_{n=2}^{N_t} \tau \| \nabla (\tilde{\mathbf{u}}_{r}^{n+1}-\tilde{\mathbf{u}}_{h}^{n+1}) \|_0^2
			+\frac{4\tau^2}{3} \sum_{n=2}^{N_t} \tau \|\nabla (p_{r}^{n+1}-p_{h}^{n+1})\|_0^2  \nonumber \\ 
			&\leq  C(\alpha) e^{T \frac{7 \tau + 4\alpha}{2}} \Big{(} \sum_{i=N_{\tilde{u}} +1}^{M_{\tilde{u}}} \tilde{\lambda}_i( \nu \tau^{-2} \| \bm{\tilde{\varphi}}_i\|_0^2 + \|\nabla \bm{\tilde{\varphi}}_i\|_0^2) + \sum_{i=N_p+1}^{M_p} \sigma_i( \nu \| \psi_i \|_0^2 + \| \nabla \psi_i \|_0^2 )\Big{)},
		\end{align}
		where the constant $C(\alpha)$ is independent from the discretization and reduced parameters.
	\end{theorem}
	\begin{proof}
Denote
		\begin{equation}
			\begin{split}
				&\mathbf{u}_{r}^{n+1}-\mathbf{u}_{h}^{n+1}
				=\mathbf{u}_{r}^{n+1}-\Pi_{\mathbf{v}}\mathbf{u}_{h}^{n+1}
				+\Pi_{\mathbf{v}}\mathbf{u}_{h}^{n+1}-\mathbf{u}_{h}^{n+1}
				:=e_{\mathbf{u}}^{n+1}+\theta_{\mathbf{u}}^{n+1}, \\
                &\mathbf{\tilde{u}}_{r}^{n+1}-\mathbf{\tilde{u}}_{h}^{n+1}
				=\mathbf{\tilde{u}}_{r}^{n+1}-\Pi_{\mathbf{\tilde{v}}}\mathbf{\tilde{u}}_{h}^{n+1}
				+\Pi_{\mathbf{\tilde{v}}}\mathbf{\tilde{u}}_{h}^{n+1}-\mathbf{\tilde{u}}_{h}^{n+1}
				:=e_{\mathbf{\tilde{u}}}^{n+1}+\theta_{\mathbf{\tilde{u}}}^{n+1}, \\
				&p_{r}^{n+1}-p_{h}^{n+1}
				=p_{r}^{n+1}-\Pi_{q} p_{h}^{n+1}+\Pi_{q} p_{h}^{n+1}-p_{h}^{n+1}
				:=e_{p}^{n+1}+\theta_{p}^{n+1}.
			\end{split}
		\end{equation}
    First, for the projection of the velocity $\mathbf{u}_h^{n+1},$ from \eqref{fom_bdf_2b} we obtain:
    \begin{equation}\label{eq:igualROM}
        (\Pi_{\mathbf{v}} \mathbf{u}_h^{n+1}, \mathbf{v}_r) = (\Pi_{\mathbf{\tilde{v}}} \mathbf{\tilde{u}}_h^{n+1}, \mathbf{v}_r) + (\theta_{\mathbf{u}}^{n+1},\mathbf{v}_r) -   (\theta_{\mathbf{\tilde{u}}}^{n+1},\mathbf{v}_r) \hspace{0.5cm} \forall \mathbf{v}_r \in U_r.
    \end{equation}
    If we substract equation \eqref{eq:igualROM} from the equation \eqref{rom_bdf_2c} we get:
    \begin{equation}
        (e_{\mathbf{u}}^{n+1},\mathbf{v}_r)  = (e_{\mathbf{\tilde{u}}}^{n+1},\mathbf{v}_r) - (\theta_{\mathbf{u}}^{n+1},\mathbf{v}_r) +   (\theta_{\mathbf{\tilde{u}}}^{n+1},\mathbf{v}_r) \hspace{0.5cm} \forall \mathbf{v}_r \in U_r.
    \end{equation}
    By taking $\mathbf{v}_r = e_{\mathbf{u}}^{n+1}$ and using Cauchy-Schwarz and Young inequalities, it follows
    \begin{equation}
       \dfrac14 \| e_{\mathbf{u}}^{n+1} \|_0^2 \leq \| e_{\mathbf{\tilde{u}}}^{n+1} \|_0^2 + \| \theta_{\mathbf{u}}^{n+1} \|_0^2 + \| \theta_{\mathbf{\tilde{u}}}^{n+1} \|_0^2.
    \end{equation}
    If we multiply last equation by $\tau$, summing from $n = 2$ to $k \leq N_t$ and using the triangle inequality with the POD projection error estimate, we get
    \begin{equation}
        \sum_{n=2}^{N_t} \tau \| \mathbf{u}_h^{n+1} - \mathbf{u}_r^{n+1} \|_0^2 \leq 4 \sum_{n=2}^{N_t} \tau \|e_{\tilde{\mathbf{u}}}^{n+1}\|_0^2 + 4\sum_{i=N_{\tilde{u}} +1}^{M_{\tilde{u}}} \tilde{\lambda}_i \| \bm{\tilde{\varphi}}_i\|_0^2 +  5\sum_{i=N_u + 1}^{M_u} \lambda_i \| \bm{\varphi}_i\|_0^2.
    \end{equation}
    Then, estimate \eqref{rom_ea_th_0} follows from the last inequality by proving \eqref{rom_ea_th}. Therefore, we continue the proof by proving \eqref{rom_ea_th}.
    
	Taking the difference of \eqref{fom_bdf_1} and \eqref{rom_bdf_1} we get
		\begin{equation}\label{rom_ea_pf_1}
			\begin{split}
				(\partial^2e_{\mathbf{\tilde{u}}}^{n+1},\mathbf{v})+ (\partial^2 \theta_{\mathbf{\tilde{u}}}^{n+1},\mathbf{v})
				+\nu (\nabla e_{\mathbf{\tilde{u}}}^{n+1},\nabla\mathbf{v})
				+ \nu(\nabla \theta_{\mathbf{\tilde{u}}}^{n+1},\nabla v) \\
				+\frac{1}{3}(\nabla (7e_{p}^{n}-5e_{p}^{n-1}+e_{p}^{n-2}), \mathbf{v})
				+\frac{1}{3}(\nabla (7\theta_{p}^{n}-5\theta_{p}^{n-1}+\theta_{p}^{n-2}),\mathbf{v})=0,
			\end{split}
		\end{equation}
		\begin{equation}\label{rom_ea_pf_2}
			\begin{split}
				(\nabla \cdot e_{\mathbf{\tilde{u}}}^{n+1}, q)+ (\nabla \cdot \theta_{\mathbf{\tilde{u}}}^{n+1}, q)
				+\frac{2\tau}{3}(\nabla(e_{p}^{n+1}-e_{p}^{n}),\nabla q) + \frac{2 \tau}{3}(\nabla(\theta_p^{n+1}- \theta_p^{n}),\nabla q)
				=0.
			\end{split}
		\end{equation}
		Taking $\mathbf{v}=e_{\mathbf{\tilde{u}}}^{n+1}$ and using the equality
		$2a(3a-4b+c)=a^2-b^2+(2a-b)^2-(2b-c)^2+(a-2b+c)^2,$ we have
		\begin{equation}\label{rom_ea_pf_3}
			\begin{split}
				&\frac{1}{4\tau}(\|e_{\mathbf{\tilde{u}}}^{n+1}\|_{0}^2-\|e_{\mathbf{\tilde{u}}}^{n}\|_{0}^2
				+\|2e_{\mathbf{\tilde{u}}}^{n+1}-e_{\mathbf{\tilde{u}}}^{n}\|_{0}^2-\|2e_{\mathbf{\tilde{u}}}^{n}-e_{\mathbf{\tilde{u}}}^{n-1}\|_{0}^2
				+\|e_{\mathbf{\tilde{u}}}^{n+1}-2e_{\mathbf{\tilde{u}}}^{n}+e_{\mathbf{\tilde{u}}}^{n-1}\|_{0}^2)
				+\nu \|\nabla e_{\mathbf{\tilde{u}}}^{n+1}\|_{0}^2 \nonumber \\
				&+(\partial^2\theta_{\mathbf{\tilde{u}}}^{n+1},e_{\mathbf{\tilde{u}}}^{n+1})
				+\nu (\nabla \theta_{\mathbf{\tilde{u}}}^{n+1}, \nabla e_{\mathbf{\tilde{u}}}^{n+1})
				+\frac{1}{3}(\nabla (7e_{p}^{n}-5e_{p}^{n-1}+e_{p}^{n-2}),  e_{\mathbf{\tilde{u}}}^{n+1}) \\
				&+\frac{1}{3}(\nabla (7\theta_{p}^{n}-5\theta_{p}^{n-1}+\theta_{p}^{n-2}), e_{\mathbf{\tilde{u}}}^{n+1})=0,
			\end{split}
		\end{equation}
		Taking $q=\frac{1}{3}(7e_{p}^{n}-5e_{p}^{n-1}+e_{p}^{n-2})$, we obtain
		\begin{equation}\label{rom_ea_pf_4}
			\begin{split}
				&-\frac{1}{3}(e_{\mathbf{\tilde{u}}}^{n+1}, \nabla (7e_{p}^{n}-5e_{p}^{n-1}+e_{p}^{n-2}))
				-\frac{1}{3}(\theta_{\mathbf{\tilde{u}}}^{n+1}, \nabla (7e_{p}^{n}-5e_{p}^{n-1}+e_{p}^{n-2}))\\
				&+\frac{2\tau}{9}(\nabla(e_{p}^{n+1}-e_{p}^{n}),\nabla (7e_{p}^{n}-5e_{p}^{n-1}+e_{p}^{n-2}))
				+\frac{2\tau}{9}(\nabla(\theta_{p}^{n+1}-\theta_{p}^{n}),\nabla (7e_{p}^{n}-5e_{p}^{n-1}+e_{p}^{n-2}))
				=0.
			\end{split}
		\end{equation}
		For \eqref{rom_ea_pf_4} at time $t^{n}$ and $t^{n-1}$, we have
		\begin{equation}\label{rom_ea_pf_5}
			\begin{split}
				(\nabla \cdot e_{\mathbf{\tilde{u}}}^{n}, q)+ (\nabla \cdot \theta_{\mathbf{\tilde{u}}}^{n}, q)
				+\frac{2\tau}{3}(\nabla(e_{p}^{n}-e_{p}^{n-1}),\nabla q)
				+\frac{2\tau}{3}(\nabla(\theta_{p}^{n}-\theta_{p}^{n-1}),\nabla q)
				=0.
			\end{split}
		\end{equation}
		\begin{equation}
			\label{rom_ea_pf_4_0}
			(\nabla \cdot e_{\mathbf{\tilde{u}}}^{n-1}, q)+ (\nabla \cdot \theta_{\mathbf{\tilde{u}}}^{n-1}, q)
			+\frac{2\tau}{3}(\nabla(e_{p}^{n-1}-e_{p}^{n-2}),\nabla q)
			+\frac{2\tau}{3}(\nabla(\theta_{p}^{n-1}-\theta_{p}^{n-2}),\nabla q)
			=0.
		\end{equation}
		Subtracting equation \eqref{rom_ea_pf_4} from equation \eqref{rom_ea_pf_5} and \eqref{rom_ea_pf_4_0}
		\begin{equation}\label{rom_ea_pf_6}
			\begin{split}
				(\nabla \cdot (3e_{\mathbf{\tilde{u}}}^{n+1}-4e_{\mathbf{\tilde{u}}}^{n}+e_{\mathbf{\tilde{u}}}^{n-1}), q)
				+(\nabla \cdot (3\theta_{\mathbf{\tilde{u}}}^{n+1}-4\theta_{\mathbf{\tilde{u}}}^{n}+\theta_{\mathbf{\tilde{u}}}^{n-1}), q) \nonumber \\
				+\frac{2\tau}{3}(\nabla(3e_{p}^{n+1}-7e_{p}^{n}+5e_{p}^{n-1}-e_{p}^{n-2}),\nabla q) \\
				+\frac{2\tau}{3}(\nabla(3\theta_{p}^{n+1}-7\theta_{p}^{n}+5\theta_{p}^{n-1}-\theta_{p}^{n-2}),\nabla q)
				=0.
			\end{split}
		\end{equation}
		and setting $q=e_{p}^{n+1}-e_{p}^{n}$, we obtain
		\begin{align}\label{rom_ea_pf_7}
			(\nabla \cdot (3e_{\mathbf{\tilde{u}}}^{n+1}-4e_{\mathbf{\tilde{u}}}^{n}+e_{\mathbf{\tilde{u}}}^{n-1}), e_{p}^{n+1}-e_{p}^{n})
			+(\nabla \cdot (3\theta_{\mathbf{\tilde{u}}}^{n+1}-4\theta_{\mathbf{\tilde{u}}}^{n}+\theta_{\mathbf{\tilde{u}}}^{n-1}), e_{p}^{n+1}-e_{p}^{n}) \nonumber \\
			+\frac{2\tau}{3}(\nabla(3e_{p}^{n+1}-7e_{p}^{n}+5e_{p}^{n-1}-e_{p}^{n-2}),\nabla (e_{p}^{n+1}-e_{p}^{n})) \\
			+\frac{2\tau}{3}(\nabla(3\theta_{p}^{n+1}-7\theta_{p}^{n}+5\theta_{p}^{n-1}-\theta_{p}^{n-2}),\nabla (e_{p}^{n+1}-e_{p}^{n}))
			=0.
		\end{align}
		Combining all results, the following equality holds
		\begin{align}\label{rom_ea_pf_8}
			&\frac{1}{4\tau}(\|e_{\mathbf{\tilde{u}}}^{n+1}\|_{0}^2-\|e_{\mathbf{\tilde{u}}}^{n}\|_{0}^2
			+\|2e_{\mathbf{\tilde{u}}}^{n+1}-e_{\mathbf{\tilde{u}}}^{n}\|_{0}^2-\|2e_{\mathbf{\tilde{u}}}^{n}-e_{\mathbf{\tilde{u}}}^{n-1}\|_{0}^2
			+\|e_{\mathbf{\tilde{u}}}^{n+1}-2e_{\mathbf{\tilde{u}}}^{n}+e_{\mathbf{\tilde{u}}}^{n-1}\|_{0}^2)
			+\nu \|\nabla e_{\mathbf{\tilde{u}}}^{n+1}\|_{0}^2  \nonumber \\
			&+\frac{\tau}{3}(\|\nabla e_{p}^{n+1}\|_{0}^2-\|\nabla e_{p}^{n}\|_{0}^2+\|\nabla (e_{p}^{n+1}-e_{p}^{n})\|_{0}^2) \nonumber\\
			=&-(\partial^2\theta_{\mathbf{\tilde{u}}}^{n+1},e_{\mathbf{\tilde{u}}}^{n+1}) -\nu (\nabla \theta_{\mathbf{\tilde{u}}}^{n+1}, \nabla e_{u}^{n+1})
			-\frac{1}{3}(
			\nabla (7\theta_{p}^{n}-5\theta_{p}^{n-1}+\theta_{p}^{n-2}),e_{\mathbf{\tilde{u}}}^{n+1})\\
			&+\frac{1}{3}( \nabla (7e_{p}^{n}-5e_{p}^{n-1}+e_{p}^{n-2}),\theta_{\mathbf{\tilde{u}}}^{n+1})\nonumber
			-\frac{ 2 \tau}{9} (\nabla(\theta_p^{n+1} - \theta_p^n),\nabla (7e^n_p - 5e^{n-1}_p + e^{n-2}_p)) \\
			&+\frac{1}{3}(3e_{\mathbf{\tilde{u}}}^{n+1}-4e_{\mathbf{\tilde{u}}}^{n}+e_{\mathbf{\tilde{u}}}^{n-1},\nabla ( e_{p}^{n+1}-e_{p}^{n}))
			\nonumber 
			+\frac{1}{3}(3\theta_{\mathbf{\tilde{u}}}^{n+1}-4\theta_{\mathbf{\tilde{u}}}^{n}+\theta_{\mathbf{\tilde{u}}}^{n-1},\nabla (e_{p}^{n+1}-e_{p}^{n})) \\
			&-\frac{ 2 \tau}{9} (\nabla(e_p^{n+1} - e_p^n),\nabla (3\theta^{n+1}_p - 7 \theta_p^n + 5\theta^{n-1}_p - \theta^{n-2}_p))\nonumber  :=\sum_{i=1}^{8}A_{i}.
		\end{align}
		Now we estimate the right hand side of \eqref{rom_ea_pf_8}. By using Cauchy-Schwarz inequality, Young inequality, Poincar\'e inequality and Lemma \ref{lemaSaturation}, using the spaces $Y = U_r$ and $Z = \text{span} \{\nabla \psi_1, \ldots,\nabla \psi_r\}$, we have
		\begin{align}\label{rom_ea_pf_9}
			A_1 \leq \frac{\epsilon_1^{-1}}{4} \| \delta_t \theta_{\mathbf{\tilde{u}}}^{n+1} \|_0^2 + C_p^2 \epsilon_1 \| \nabla e_{\mathbf{\tilde{u}}}^{n+1} \|_0^2,
		\end{align}
		
		\begin{align}\label{rom_ea_pf_10}
			A_2 \leq \frac{\epsilon_2^{-1} \nu}{4} \| \nabla \theta_{\mathbf{\tilde{u}}}^{n+1} \|_0^2 + \nu \epsilon_2 \| \nabla e_{\mathbf{\tilde{u}}}^{n+1} \|_0^2,
		\end{align}
		\begin{align}\label{rom_ea_pf_11}
			A_3 \leq \frac{\epsilon_3^{-1}}{12} \| 7 \theta_p^n - 5\theta_p^{n-1} +\theta_p^{n-2} \|_0^2 + \frac{\epsilon_3}{3} \| \nabla e_{\mathbf{\tilde{u}}}^{n+1} \|_0^2,
		\end{align}
		
		\begin{align}\label{rom_ea_pf_12}
			A_4 \leq \frac{8}{3 \tau^2} \| \theta_{\mathbf{\tilde{u}}}^{n+1} \|_0^2 + \tau^2 (\frac14 \| \nabla e_p^n \|_0^2 + \frac13 \| \nabla (e^n_p - e_p^{n-1} )\|_0^2 + \frac{1}{12} \| \nabla (e^{n-2}_p - e_p^{n-1}) \|_0^2),
		\end{align}
		
		\begin{align}\label{rom_ea_pf_13}
			A_5 \leq \frac89 \| \nabla(\theta_p^{n+1} - \theta_p^n)\|_0^2 + \frac{\tau^2}{3}\| \nabla e_p^n \|_0^2 + \frac{4 \tau^2 }{9} \| \nabla(e_p^n - e_p^{n-1}) \|_0^2 + \frac{\tau^2}{9} \| \nabla (e_p^{n-2} - e_p^{n-1}) \|_0^2,
		\end{align}

		\begin{align}\label{rom_ea_pf_14}
			A_6 \leq \frac{\tau \alpha}{3} \| \partial^2 e_{\mathbf{\tilde{u}}}^{n+1} \|_0^2 + \frac{\tau \alpha }{3} \| \nabla (e_p^{n} - e_p^{n-1}) \|_0^2, 
		\end{align}
		
		\begin{align}\label{rom_ea_pf_15}
			A_7 \leq \frac13 \| \partial^2 \theta_{\mathbf{\tilde{u}}}^{n+1} \|_0^2 + \frac{\tau^2}{3} \| \nabla (e_p^{n+1} - e_p^n) \|_0^2,
		\end{align}

		\begin{align}\label{rom_ea_pf_16}
			A_8 \leq  \frac19\| \nabla (3 \theta^{n+1}_p - 7 \theta_p^{n} + 5 \theta_p^{n-1} - \theta_p^{n-2}) \|_0^2 + \frac{\tau^2}{9} \| \nabla (e_p^{n+1} - e_p^n) \|_0^2.
		\end{align}
		
		Finally, using the estimates from \eqref{rom_ea_pf_9} to \eqref{rom_ea_pf_16}, we have 
		
		\begin{align}\label{rom_ea_pf_17}
			&\frac{1}{4\tau}(\|e_{\mathbf{\tilde{u}}}^{n+1}\|_{0}^2-\|e_{\mathbf{\tilde{u}}}^{n}\|_{0}^2
			+\|2e_{\mathbf{\tilde{u}}}^{n+1}-e_{\mathbf{\tilde{u}}}^{n}\|_{0}^2-\|2e_{\mathbf{\tilde{u}}}^{n}-e_{\mathbf{\tilde{u}}}^{n-1}\|_{0}^2)
			+\nu \|\nabla e_{\mathbf{\tilde{u}}}^{n+1}\|_{0}^2  \nonumber \\
			&+\frac{\tau}{3}(\|\nabla e_{p}^{n+1}\|_{0}^2-\|\nabla e_{p}^{n}\|_{0}^2+\|\nabla (e_{p}^{n+1}-e_{p}^{n})\|_{0}^2) \leq \nonumber\\
			& (\frac{\epsilon_1^{-1}}{4} + \frac13) \| \partial^2 \theta_{\mathbf{\tilde{u}}}^{n+1} \|_0^2 + \frac{\epsilon_2^{-1} \nu}{4} \| \nabla \theta_{\mathbf{\tilde{u}}}^{n+1} \|_0^2 + (C_p^2 \epsilon_1 + \nu \epsilon_2 + \frac{\epsilon_3}{3}) \| \nabla e_{\mathbf{\tilde{u}}}^{n+1} \|_0^2 + \frac{\epsilon_3^{-1}}{12} \| 7 \theta_p^n - 5\theta_p^{n-1} + \theta_p^{n-2} \|_0^2 \nonumber \\
			& + \frac{8}{3 \tau^2} \| \theta_{\mathbf{\tilde{u}}}^{n+1} \|_0^2 + (\frac{\tau^2}{4} + \frac{\tau^2}{3}) \| \nabla e_p^n \|_0^2 + (\frac{\tau^2}{3} + \frac{4 \tau^2}{9}) \| \nabla (e_p^n - e_p^{n-1}) \|_0^2 + (\frac{\tau^2}{12}+ \frac{\tau^2}{9}) \| \nabla (e_p^{n-2} - e_p^{n-1}) \|_0^2\nonumber \\
			&+  \frac{\tau \alpha}{3} \| \partial^2 e_{\mathbf{\tilde{u}}}^{n+1} \|_0^2 
			+ (\frac{\tau \alpha + \tau^2}{3} + \frac{\tau^2}{9}) \| \nabla (e^{n+1}_p - e^{n}_p) \|_0^2 + \frac{8}{9} \| \nabla (\theta_p^{n+1} - \theta_p^n) \|_0^2   \nonumber \\
			& + \frac19\| \nabla (3\theta_p^{n+1} - 7\theta_p^n + 5 \theta_p^{n-1} - \theta_p^{n-2}) \|_0^2. 
		\end{align}
		Taking in \eqref{rom_ea_pf_17} $\epsilon_1 = \frac{\nu}{6 C_p^2}$, $\epsilon_2 = \frac{1}{6}$ and $\epsilon_3 = \frac{\nu}{2}$ and  multiplying by $4\tau,$ we obtain:

		\begin{align}\label{rom_ea_pf_18}
			&\|e_{\mathbf{\tilde{u}}}^{n+1}\|_{0}^2-\|e_{\mathbf{\tilde{u}}}^{n}\|_{0}^2
			+\|2e_{\mathbf{\tilde{u}}}^{n+1}-e_{\mathbf{\tilde{u}}}^{n}\|_{0}^2-\|2e_{\mathbf{\tilde{u}}}^{n}-e_{\mathbf{\tilde{u}}}^{n-1}\|_{0}^2 
			+ \| e_{\mathbf{\tilde{u}}}^{n+1} - 2 e_{\mathbf{\tilde{u}}}^n + e_{\mathbf{\tilde{u}}}^{n-1}\|_0^2
			\nonumber \\
			&  
			+\frac{4 \tau \nu}{2} \|\nabla e_{\mathbf{\tilde{u}}}^{n+1}\|_{0}^2+\frac{4 \tau^2}{3}(\|\nabla e_{p}^{n+1}\|_{0}^2-\|\nabla e_{p}^{n}\|_{0}^2+\|\nabla (e_{p}^{n+1}-e_{p}^{n})\|_{0}^2) \leq \nonumber\\
			& 4 \tau(\frac{6 C_p^2}{4 \nu} + \frac13) \| \partial^2 \theta_{\mathbf{\tilde{u}}}^{n+1} \|_0^2 + \frac{ 4 \tau \nu}{24} \| \nabla \theta_{\mathbf{\tilde{u}}}^{n+1} \|_0^2 + \frac{4 \tau\nu}{24} \| 7 \theta_p^n - 5\theta_p^{n-1} + \theta_p^{n-2} \|_0^2 \nonumber \\
			& + \frac{32}{3 \tau} \| \theta_{\mathbf{\tilde{u}}}^{n+1} \|_0^2 + 4 \tau(\frac{\tau^2}{4} + \frac{\tau^2}{3}) \| \nabla e_p^n \|_0^2 + 4 \tau (\frac{\tau^2}{3} + \frac{4 \tau^2}{9}) \| \nabla (e_p^n - e_p^{n-1}) \|_0^2  \nonumber \\
			&+ 4 \tau(\frac{\tau^2}{12}+ \frac{\tau^2}{9}) \| \nabla (e_p^{n-2} - e_p^{n-1}) \|_0^2 
			+ 4\tau(\frac{\tau \alpha + \tau^2}{3} + \frac{\tau^2}{9}) \| \nabla (e^{n+1}_p - e^{n}_p) \|_0^2 \nonumber \\
			&+ 4 \tau \frac{8}{9} \| \nabla (\theta_p^{n+1} - \theta_p^n) \|_0^2 +\frac{ 4 \tau}{9} \| \nabla (3\theta_p^{n+1} - 7\theta_p^n + 5 \theta_p^{n-1} - \theta_p^{n-2}) \|_0^2 + \dfrac{4 \tau^2 \alpha}{3} \| \partial^2 e_{\mathbf{\tilde{u}}}^{n+1} \|_0^2 . 
		\end{align}
		Summing \eqref{rom_ea_pf_18} from $n=2$ to $k\leq N_t$ and following the same procedure as in Lemma 3 of  \cite{JuliaBoscoBDF2}, we have: 
		\begin{align}\label{rom_ea_pf_19}
			&\|e_{\mathbf{\tilde{u}}}^{k+1}\| + \frac{4\tau \nu}{2} \sum_{n=2}^k \| \nabla e_{\mathbf{\tilde{u}}}^{n+1} \|_0^2 + \frac{4\tau^2}{3} \| \nabla e_p^{k+1} \|  + \dfrac{4\tau^2(1 - \frac{17\tau}{4} - \alpha)}{3} \| \nabla (e_p^{n+1} - e_p^n) \|_0^2 \leq  \nonumber\\
			&7\|e_{\mathbf{\tilde{u}}}^2\|_0^2 + 3\|e_{\mathbf{\tilde{u}}}^1\|_0^2  + \frac{4 \tau^2}{3} \sum_{i=0}^{2}\| \nabla e_p^{i}\|_0^2 + 4 \tau (\frac{\tau^2}{4} + \frac{\tau^2}{3}) \sum_{n=2}^k \| \nabla e_p^n \|_0^2  + \dfrac{4 \tau^2 \alpha}{3} \sum_{n=2}^k \| \partial^2 e_{\mathbf{\tilde{u}}}^{n+1} \|_0^2 \nonumber \\
			& +4(\frac{6C_p^2}{4\nu} + \frac13) \sum_{n=2}^k \tau \| \partial^2 \theta_{\mathbf{\tilde{u}}}^{n+1} \|_0^2 + \frac{4\nu}{24} \sum_{n=2}^k \tau \| \nabla \theta_{\mathbf{\tilde{u}}}^{n+1} \|_0^2 + \frac{4\nu}{24} \sum_{n=2}^k \tau \| 7\theta_p^n - 5 \theta_p^{n-1} + \theta_p^{n-2} \|_0^2 \nonumber \\
			& \frac{32}{3\tau^2} \sum_{n=2}^{k} \tau \| \theta_{\mathbf{\tilde{u}}}^{n+1} \|_0^2 + \frac{32}{9} \sum_{n=2}^k \tau \|\nabla(\theta_p^{n+1} - \theta_p^n) \|_0^2 + \frac49 \sum_{n=2}^k \tau \| \nabla (3 \theta_p^{n+1} - 7 \theta_p^n +5 \theta_p^{n-1} - \theta_p^{n-2}) \|_0^2.
		\end{align}

		\textit{Bounding the time derivative of the error.} We estimate the time derivative of the error using the same idea as in \cite[Theorem~2]{GarciaArchillaJohnNovo2025}. First, we consider $\mathbf{v} = \partial^2 e_{\mathbf{\tilde{u}}}^{n+1} $ in \eqref{rom_ea_pf_1}: 
		\begin{equation}\label{rom_g_pf_1}
			\begin{split}
				\| \partial^2 e_{\mathbf{\tilde{u}}}^{n+1} \|_0^2 +  (\partial^2 \theta_{\mathbf{\tilde{u}}}^{n+1},\partial^2 e_{\mathbf{\tilde{u}}}^{n+1})
				+\nu (\nabla e_{\mathbf{\tilde{u}}}^{n+1},\nabla \partial^2 e_{\mathbf{\tilde{u}}}^{n+1})
				+ \nu(\nabla \theta_{\mathbf{\tilde{u}}}^{n+1},\partial^2 e_{\mathbf{\tilde{u}}}^{n+1}) \\
				+\frac{1}{3}(\nabla (7e_{p}^{n}-5e_{p}^{n-1}+e_{p}^{n-2}), \partial^2 e_{\mathbf{\tilde{u}}}^{n+1})
				+\frac{1}{3}(\nabla (7\theta_{p}^{n}-5\theta_{p}^{n-1}+\theta_{p}^{n-2}),\partial^2 e_{\mathbf{\tilde{u}}}^{n+1})=0,
			\end{split}
		\end{equation} 
		Then, by using the same equality as in \eqref{rom_ea_pf_3}, we obtain:
		\begin{align}\label{rom_g_pf_2}
			4 \tau \| \partial^2 e_{\mathbf{\tilde{u}}}^{n+1} \|_0^2 + \nu \| \nabla e_{\mathbf{\tilde{u}}}^{n+1} \|_0^2 + \nu\| \nabla (2e^{n+1}_{\mathbf{\tilde{u}}} - e^{n}_{\mathbf{\tilde{u}}} ) \|_0^2 + \nu\| \nabla (e_{\mathbf{\tilde{u}}}^{n+1} - 2 e_{\mathbf{\tilde{u}}}^n + e_{\mathbf{\tilde{u}}}^{n-1}) \|_0^2 =  \nonumber \\ \nu\| \nabla e_{\mathbf{\tilde{u}}}^n \|_0^2 + \nu\| \nabla (2e^{n}_{\mathbf{\tilde{u}}} - e^{n-1}_{\mathbf{\tilde{u}}} ) \|_0^2 + 4\tau (A, \partial^2 e_{\mathbf{\tilde{u}}}^{n+1}),
		\end{align}
		where 
		\begin{equation*}
			(A, \partial^2 e_{\mathbf{\tilde{u}}}^{n+1}) =  (- \partial^2 \theta_{\mathbf{\tilde{u}}}^{n+1}  - \nu \theta_{\mathbf{\tilde{u}}}^{n+1} - \frac13 \nabla(7e_p^n - 5e_p^{n-1} + e_p^{n-2}) - \frac13 \nabla(7 \theta_p^n - 5\theta_p^{n-1} + \theta_p^{n-2}), \partial^2 e_{\mathbf{\tilde{u}}}^{n+1}).
		\end{equation*}
		
		Now, by applying Cauchy-Schwarz inequality and Young inequality in \eqref{rom_g_pf_2}:
		\begin{align}\label{rom_g_pf_3}
		 2\tau \| \partial^2 e_{\mathbf{\tilde{u}}}^{n+1} \|_0^2 + \nu\| \nabla e_{\mathbf{\tilde{u}}}^{n+1} \|_0^2 + \nu\| \nabla (2e^{n+1}_{\mathbf{\tilde{u}}} - e^{n}_{\mathbf{\tilde{u}}} ) \|_0^2 + \nu\| \nabla (e_{\mathbf{\tilde{u}}}^{n+1} - 2 e_{\mathbf{\tilde{u}}}^n + e_{\mathbf{\tilde{u}}}^{n-1}) \|_0^2 \leq \nonumber \\ \nu \| \nabla e_{\mathbf{\tilde{u}}}^n \|_0^2 + \nu\| \nabla (2e^{n}_{\mathbf{\tilde{u}}} - e^{n-1}_{\mathbf{\tilde{u}}} ) \|_0^2 + 2\tau \| A \|_0^2. 
		\end{align}
		
		Therefore, summing from $n = 2$ to $k \leq N_t$ in \eqref{rom_g_pf_3} and simplifying, using the same procedure as in Lemma 3 of \cite{JuliaBoscoBDF2}:
		\begin{align}\label{rom_g_pf_4}
			 \sum_{n=2}^k 2\tau \| \partial^2 e_{\mathbf{\tilde{u}}}^{n+1} \|_0^2 \leq \nu (7\|\nabla e_{\mathbf{\tilde{u}}}^2\|_0^2 + 3 \| \nabla e_{\mathbf{\tilde{u}}}^1 \|_0^2) + \sum_{n=2}^k 2 \tau \| A \|_0^2.
		\end{align}
		
		Once the temporal derivative of the error is bounded, we use the following estimate: 
		\begin{align*}
		 \tau\| A \|_0^2 \leq \tau \|\partial^2 \theta_{\mathbf{\tilde{u}}}^{n+1} \|_0^2 + \tau \nu \| \nabla \theta_{\mathbf{\tilde{u}}}^{n+1} \|_0^2 + \tau \|\nabla e_p^n \|_0^2 + \dfrac{4\tau}{3} \| \nabla (e_p^n - e_p^{n-1}) \|_0^2  + \\
			\dfrac{\tau}{3} \| \nabla (e_p^{n-1} - e_p^{n-2})\|_0^2 + \dfrac{\tau}{3} \| \nabla (7 \theta_p^{n} - 5 \theta_p^{n-1} + \theta_{p}^{n-2}) \|_0^2,
		\end{align*}
		and introduce \eqref{rom_g_pf_4} into \eqref{rom_ea_pf_19}, assuming that $\tau < \dfrac{4(1 - \frac{8}{3}\alpha)}{17}$:
		
		\begin{align}\label{rom_ea_pf_new}
			&\|e_{\mathbf{\tilde{u}}}^{k+1}\| + \frac{4\tau \nu}{2} \sum_{n=2}^k \| \nabla e_{\mathbf{\tilde{u}}}^{n+1} \|_0^2 + \frac{4\tau^2}{3} \| \nabla e_p^{k+1} \|  \leq 7\|e_{\mathbf{\tilde{u}}}^2\|_0^2 + 3\|e_{\mathbf{\tilde{u}}}^1\|_0^2  + \frac{4 \tau^2}{3} \sum_{i=0}^{2}\| \nabla e_p^{i}\|_0^2  \nonumber\\
			& +\dfrac{2\nu \tau \alpha }{3} (7\| \nabla e_{\mathbf{\tilde{u}}}^2\|_0^2 + 3 \| \nabla e_{\mathbf{\tilde{u}}}^1\|_0^2) + 4 \tau (\frac{\tau^2}{4} + \frac{\tau^2}{3} + \frac{\tau \alpha}{3}) \sum_{n=2}^k \| \nabla e_p^n \|_0^2 +4(\frac{6C_p^2}{4\nu} + \frac13 + \dfrac{\tau \alpha}{3}) \sum_{n=2}^k \tau \| \partial^2 \theta_{\mathbf{\tilde{u}}}^{n+1} \|_0^2  \nonumber \\
			& + \frac{4\nu(1 + \dfrac{\tau \alpha}{3})}{24} \sum_{n=2}^k \tau \| \nabla \theta_{\mathbf{\tilde{u}}}^{n+1} \|_0^2 + \frac{4\nu}{24} \sum_{n=2}^k \tau \| 7\theta_p^n - 5 \theta_p^{n-1} + \theta_p^{n-2} \|_0^2+  \frac{32}{3\tau^2} \sum_{n=2}^{k} \tau \| \theta_{\mathbf{\tilde{u}}}^{n+1} \|_0^2 \nonumber \\
			& + \frac{32}{9} \sum_{n=2}^k \tau \|\nabla(\theta_p^{n+1} - \theta_p^n) \|_0^2 + (\frac49 + \frac{4\tau \alpha}{9}) \sum_{n=2}^k \tau \| \nabla (3 \theta_p^{n+1} - 7 \theta_p^n +5 \theta_p^{n-1} - \theta_p^{n-2}) \|_0^2  =  7\|e_{\mathbf{\tilde{u}}}^2\|_0^2 + 3\|e_{\mathbf{\tilde{u}}}^1\|_0^2 \nonumber \\ 
			&  + \frac{4 \tau^2}{3} \sum_{i=0}^{2}\| \nabla e_p^{i}\|_0^2 +\dfrac{2\nu \tau \alpha }{3} (7\| \nabla e_{\mathbf{\tilde{u}}}^2\|_0^2 + 3 \| \nabla e_{\mathbf{\tilde{u}}}^1\|_0^2) +  \dfrac{7\tau + 4\alpha}{4}\sum_{n=2}^k \dfrac{4\tau^2}{3}\| \nabla e_p^n \|_0^2 +  \sum_{i=1}^6 B_i. 
		\end{align}
		Now, by applying Gronwall's lemma  (see for instance \cite{HeywoodRannacherSINUM90}, Lemma 5.1) in \eqref{rom_ea_pf_19} we get:
		
		\begin{align}\label{rom_ea_pf_20}
			&\max_{2< k \leq N_t+1} \| e_{\mathbf{\tilde{u}}}^k\|_0^2 + \frac{4 \tau \nu}{2} \sum_{n=2}^{N_t} \| \nabla e_{\mathbf{\tilde{u}}}^{n+1} \|_0^2 + \max_{2< k \leq N_t+1} \frac{4\tau^2}{3} \| \nabla e_p^{k} \|_0^2  \nonumber \\
			& \leq e^{\frac{ 7\tau + 4 \alpha }{2} T}(7 \|e_{\mathbf{\tilde{u}}}^2 \|_0^2 + 3 \|e_{\mathbf{\tilde{u}}}^1 \|_0^2 + C(\nu,\alpha)(7 \| \nabla e_{\mathbf{\tilde{u}}}^2\|_0^2 + 3 \| \nabla e_{\mathbf{\tilde{u}}}^1\|_0^2) +  \frac{4 \tau^2}{3} \sum_{i=0}^{2}\| \nabla e_p^i \|_0^2 + 2 \sum_{i=1}^6 B_i^N).
		\end{align}
		
		Therefore, using POD projection error estimates, we can bound the terms $\{B_i^{N}\}_{i=1}^6$ on the right-hand side of \eqref{rom_ea_pf_20} obtaining:
		
		\begin{align}\label{rom_ea_pf_21}
			&\max_{2< k \leq N_t+1} \| e_{\mathbf{\tilde{u}}}^k\|_0^2 + \frac{4 \tau \nu}{2} \sum_{n=2}^{N_t} \| \nabla e_{\mathbf{\tilde{u}}}^{n+1} \|_0^2 + \max_{2< k \leq N_t+1} \frac{4\tau^2}{3} \| \nabla e_p^{k} \|_0^2  \nonumber \\
			& \leq e^{\frac{ 7\tau  + 4 \alpha}{2} T}(7 \|e_{\mathbf{\tilde{u}}}^2 \|_0^2 + 3 \|e_{\mathbf{\tilde{u}}}^1 \|_0^2 + \frac{4 \tau^2}{3} \sum_{i=0}^{2}\| \nabla e_p^i \|_0^2 +  C(\nu,\alpha)(7 \| \nabla e_{\mathbf{\tilde{u}}}^2\|_0^2 + 3 \| \nabla e_{\mathbf{\tilde{u}}}^1\|_0^2) + \nonumber \\
			& C(\alpha)(\sum_{i=N_{\tilde{u}} +1}^{M_{\tilde{u}}} \tilde{\lambda}_i(\alpha \nu \tau^{-2} \| \bm{\tilde{\varphi}}_i\|_0^2) + \|\nabla \bm{\tilde{\varphi}}_i\|_0^2) +  \sum_{i=N_p+1}^{M_p} \sigma_i( \nu\| \psi_i \|_0^2 +  \| \nabla \psi_i \|_0^2 ))).
		\end{align}
		
		Finally, taking the ROM initial conditions as the POD projection of the FOM initial solutions, then \eqref{rom_ea_pf_21} reads:
		
		\begin{align}\label{rom_ea_pf_22}
			&\max_{2< k \leq N_t+1} \| e_{\mathbf{\tilde{u}}}^k\|_0^2 + \frac{4 \tau \nu}{2} \sum_{n=2}^{N_t} \| \nabla e_{\mathbf{\tilde{u}}}^{n+1} \|_0^2 + \max_{2< k \leq N_t+1} \frac{4\tau^2}{3} \| \nabla e_p^{k} \|_0^2  \nonumber \\
			& \leq C(\alpha) e^{\frac{ 7\tau  + 4 \alpha}{2} T}(
			\sum_{i=N_{\tilde{u}} +1}^{M_{\tilde{u}}} \tilde{\lambda}_i(\nu \tau^{-2} \| \bm{\tilde{\varphi}}_i\|_0^2) + \|\nabla \bm{\tilde{\varphi}}_i\|_0^2) + \sum_{i=N_p+1}^{M_p} \sigma_i( \nu \| \psi_i \|_0^2 + \| \nabla \psi_i \|_0^2 )).
		\end{align}
		Finally, using the following inequality:
		\begin{align}\label{rom_ea_pf_23}
			& \sum_{n=2}^{N_t} \tau \| e_{\mathbf{\tilde{u}}}^{n+1} \|^2 \leq T \max_{2 \leq k \leq N_t} \|e_{\mathbf{\tilde{u}}}^{k+1} \|^2, \nonumber \\
			&\sum_{n=2}^{N_t} \tau \| e_p^{n+1} \|^2 \leq T \max_{2 \leq k \leq N_t} \|e_p^{k+1} \|^2,
		\end{align}
		and the triangle inequality, we reach the error bound \eqref{rom_ea_th}.
	\end{proof}

	\begin{remark}
		We can note that we have a factor of $\tau^2$ in front of the pressure term in the stability \eqref{rom_st_1} and the POD truncation error estimation \eqref{rom_ea_th}, this can be relaxed as suggested in \cite{Azaiez2025}, Section 4.2.1.
	\end{remark}
	\newpage

	\section{The Numerical Results}
	In this section, we present some numerical results for the BDF2 time-splitting scheme introduced and analyzed in the previous section. In order to evaluate the accuracy of the full-order discretization and the performance of the ROM, we show three numerical experiments with Dirichlet boundary conditions. 
	
	First, we verify the precision of the numerical scheme by using a  solution with relatively high regularity. Next, we consider a Stokes problem in which the forcing term is chosen so that the solution exhibits reduced regularity. Although no exact solution is available in this case, this test allows us to evaluate the behaviour and robustness of the ROM under low-regularity conditions. Finally, we demonstrate some reliable results through a classical benchmark text, i.e. the lid driven cavity flow. 
    
	\subsection{Stokes tests}
    The next two tests have the same configuration. We have considered $\Omega = [0,1] \times [0,1] \subset \mathbb{R}^2$ as computational domain, the time interval is [0,T], where $T = 1,$ and the viscosity coefficient is $\nu = 1.$ The numerical method used for solving the problem is the one described in Section \eqref{FullyDiscrete}, moreover the spatial discretization used is $\mathbb{P}^2-\mathbb{P}^1$ for the pair velocity-pressure on a relatively coarse computational grid, for which we consider a uniform partition of the cavity on $100^2$ cells.
	\subsubsection{Accuracy Test}
	The accuracy test comes from a preescribed solution given in \cite{li2022efficient}:
	\begin{equation}\label{ExactSolution}
		\left\{ \begin{array}{ccl}
			u(x,y,t)  & = & \cos(t) \pi \sin(\pi x)^2 \sin(2\pi y),  \\
			v(x,y,t)  & = & \cos(t) \pi \sin(\pi x)^2 \sin(2\pi y), \\
			p(x,y,t) & = &10\cos(t)\cos(\pi x) \cos(\pi y).
		\end{array} \right. 
	\end{equation}
	 In table \eqref{tab1} we present the $L^2$ error between the exact solution and the FE solution with different time step size in order to demonstrate the second-order accuracy of the scheme. The table also displays the convergence rates, confirming that the suggested scheme achieves the theoretical accuracy. 
	\begin{table}[!h]\small
		\centering
		\caption{Time convergence of the scheme for the Stokes problem with the exact solution \eqref{ExactSolution}.}
		\begin{tabular}{c c c c c c c}
			\hline
			$\tau$
			&$\|u_{exact} - u_h\|_{L^2}$
			&$Rate$
			&$\|p_{exact} - p_h\|_{L^2}$
			&$Rate$\\
			\hline
			$1/20$          &1.35e-02  &            &4.27e-01    &        \\
			$1/40$          &3.69e-03  &1.874      &1.20e-01    &1.8300  \\
			$1/80$          &9.42e-04 &1.9705      &3.11e-02    &1.9459  \\
			$1/160$          &2.37e-04 &1.9915      &7.89e-03    &1.9812 \\
			\hline
		\end{tabular}
		\label{tab1}
	\end{table}
	
	\subsubsection{POD Reduced Order Model }
	Now, we will demonstrate the effectiveness and efficiency of the ROM described in Section \eqref{ROMScheme}  for the unsteady Stokes problem with the following forcing term
    \begin{equation}
        f(x,y;t) = \left( \begin{array}{c}
              \sqrt{(x-t)^2 + y^2} \\
              \sqrt{x^2(y-t)^2}
        \end{array} \right).
    \end{equation}
    In particular, the forcing term is designed to exhibit low regularity, allowing us to assess the robustness of the ROM in non-smooth regimes. The numerical method to get the snapshots is the previous one, using $\tau = 10^{-2}$ in the FOM simulations. For the ROM, we collect $M = 81$ snapshots for each unknown field, by storing every fourth FOM simulations in the time interval $[0.21, 1],$ in order to avoid numerical inestabilities at the beginning of the simulation. The POD modes are generated in $L^2-$norm for velocity and pressure. In the following, we will analyze the discrete relative error between the FE solution $s_h$ and the ROM solution $s_r$ in the $\ell^2(L^2)-$norm, defined as:
	\begin{equation}
		\dfrac{\|s_h - s_r\|_{\ell^2(L^2)}}{\| s_h \|_{\ell^2(L^2)}} = \sqrt{\dfrac{\sum_{i=1}^M \| s_h^i - s_r^i\|_0^2}{\sum_{i=1}^M \| s_h^i \|_0^2}}.
	\end{equation}
The decay of normalized POD singular values (left) and neglected energy (right) computed by $100 - 100 \sum_{k=1}^r \sigma_k / \sum_{k=1}^M \sigma_k,$ where $\sigma_k$ are the corresponding singular values  and $M$ the rank of the corresponding data set of the problem of each field, are shown in Figure \eqref{fig:eigenvaluesStokes}. Note that with $10$ POD basis we neglect less than $10^{-4}$ of the energy for each field. 
	\begin{figure}[ht!]
		\centering
		\begin{subfigure}[b]{0.4\textwidth}
			\includegraphics[width=\textwidth]{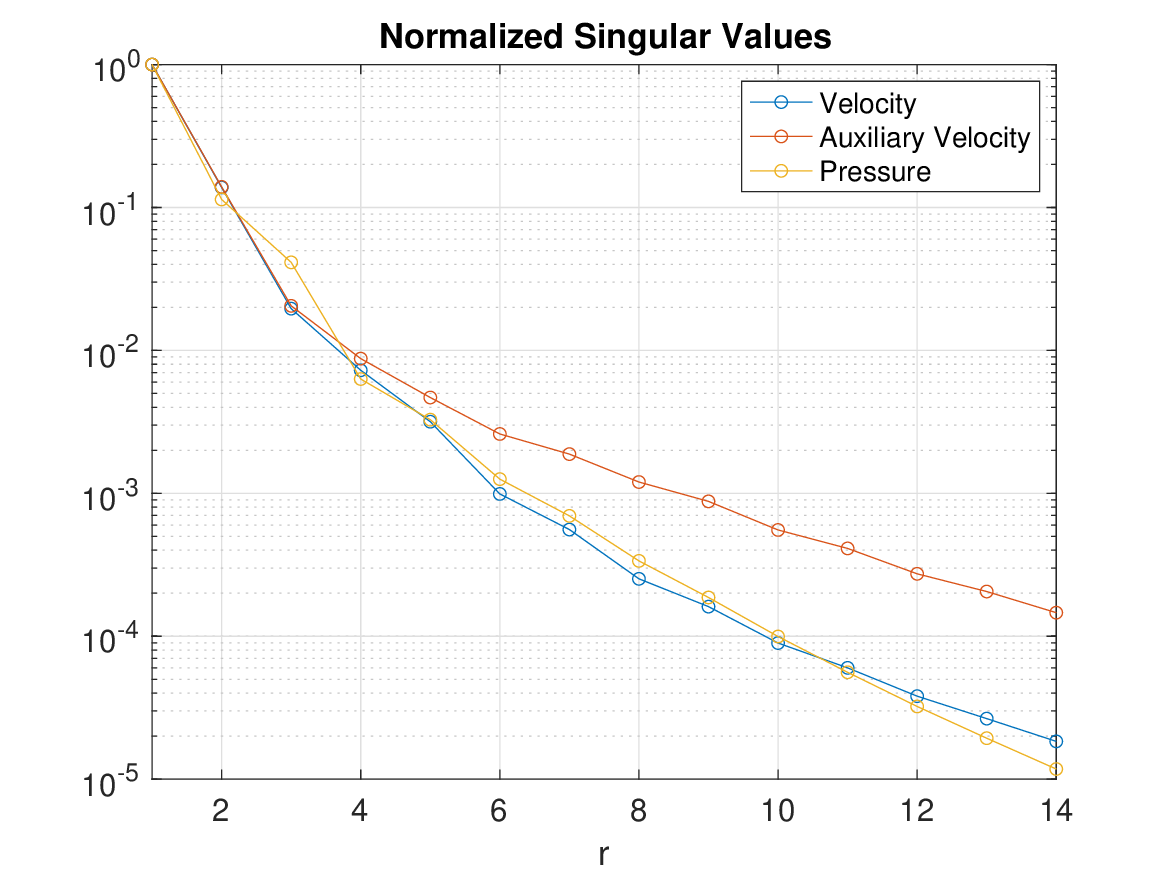}
		\end{subfigure}
		\begin{subfigure}[b]{0.4\textwidth}
			\includegraphics[width=\textwidth]{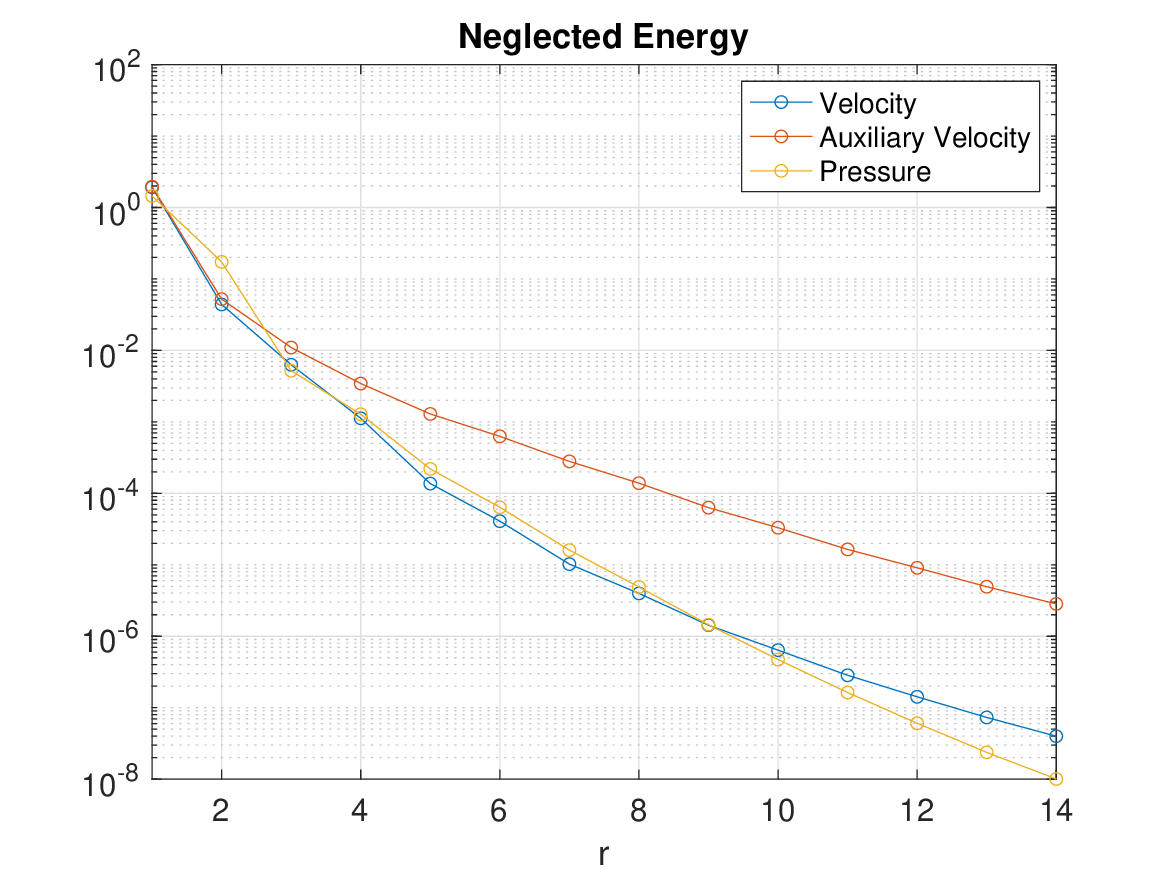}
		\end{subfigure}
		\caption{Decay of the normalized POD singular values (left) and neglected energy (right).}
		\label{fig:eigenvaluesStokes}
	\end{figure}
    
	In Figure \eqref{fig:l2l2errorStokes}, we show the $\ell^2(L^2)$ relative errors for both velocity (left) and pressure (right). 
	For comparison purposes, we also plot the $\ell^2(L^2)$ relative projection errors for both velocity and pressure in the same figure. It is worth noting that this error represents the best achievable error. 
	\begin{figure}[ht!]
		\centering
		\begin{subfigure}[b]{0.4\textwidth}
			\includegraphics[width=\textwidth]{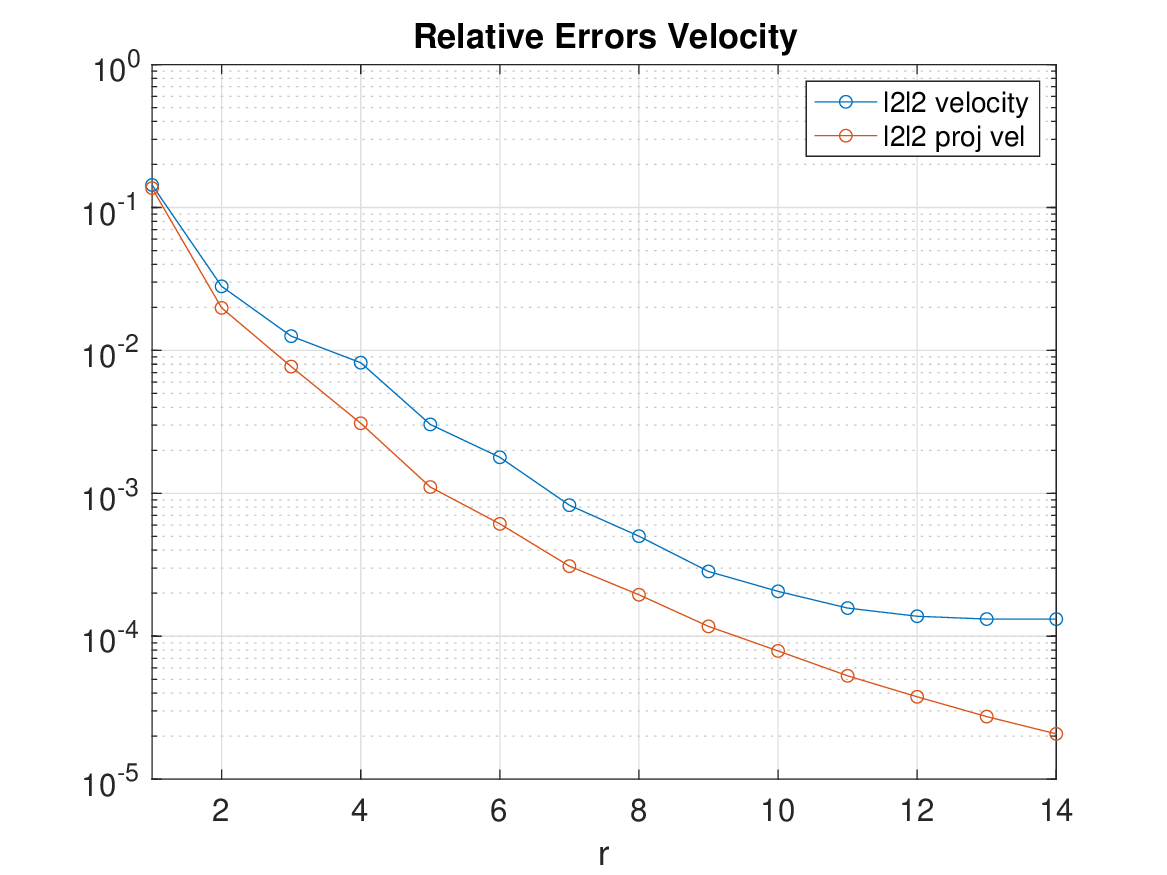}
		\end{subfigure}
		\begin{subfigure}[b]{0.4\textwidth}
			\includegraphics[width=\textwidth]{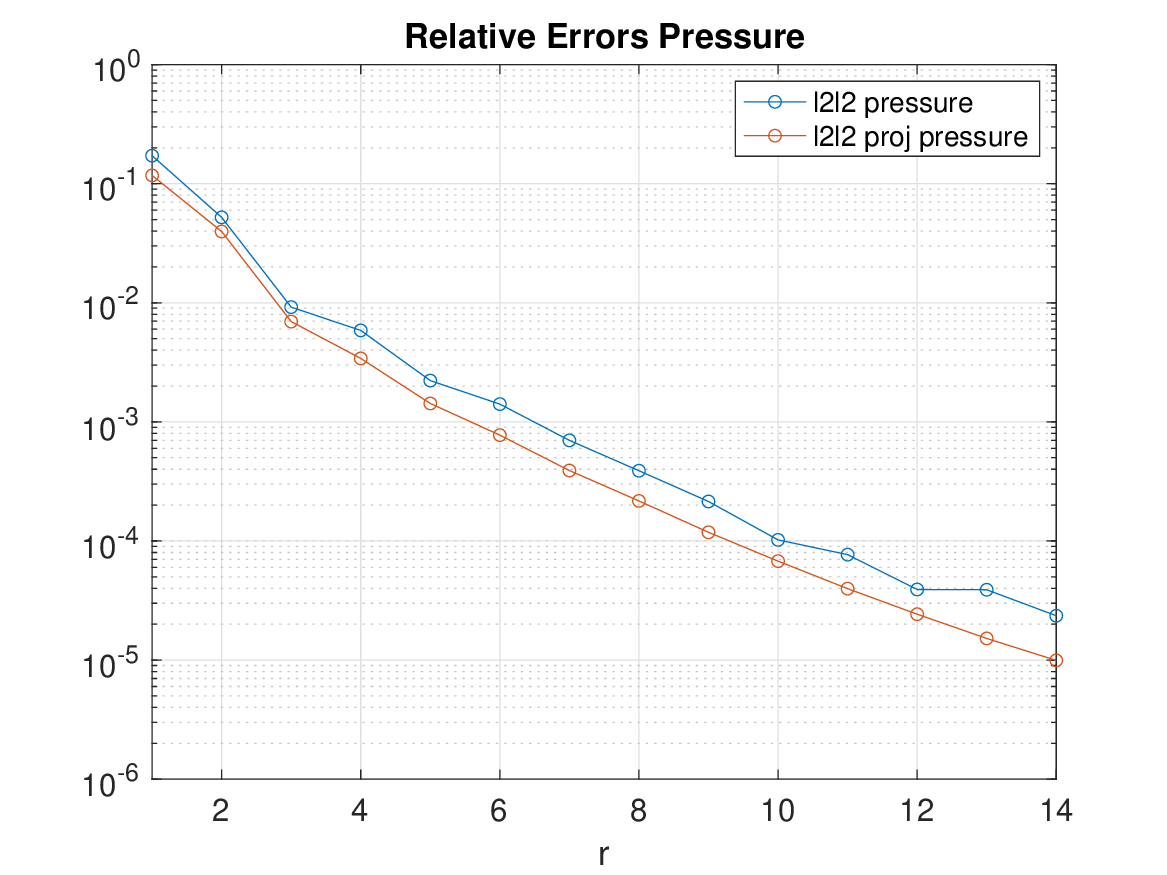}
		\end{subfigure}
		\caption{$\ell^2(L^2)$ relative errors for velocity (left) and pressure (right) depending on the number of POD modes $r$.}
		\label{fig:l2l2errorStokes}
	\end{figure}

	We can also observe in Figure \eqref{fig:l2l2errorStokes} that for the velocity case, using 14 basis functions, the relative error is around $10^{-4}$, despite the low-regularity induced by the forcing term. Using the same number of POD basis for the pressure, we achieve a lower error.

    Finally, Figure \ref{fig:theoreticalnum} shows a comparison between the theoretical error estimator \eqref{rom_ea_th} and the velocity (left) and pressure (right) error depending on the number of POD modes $r$. 

    \begin{figure}[ht!]
		\centering
		\begin{subfigure}[b]{0.4\textwidth}
			\includegraphics[width=\textwidth]{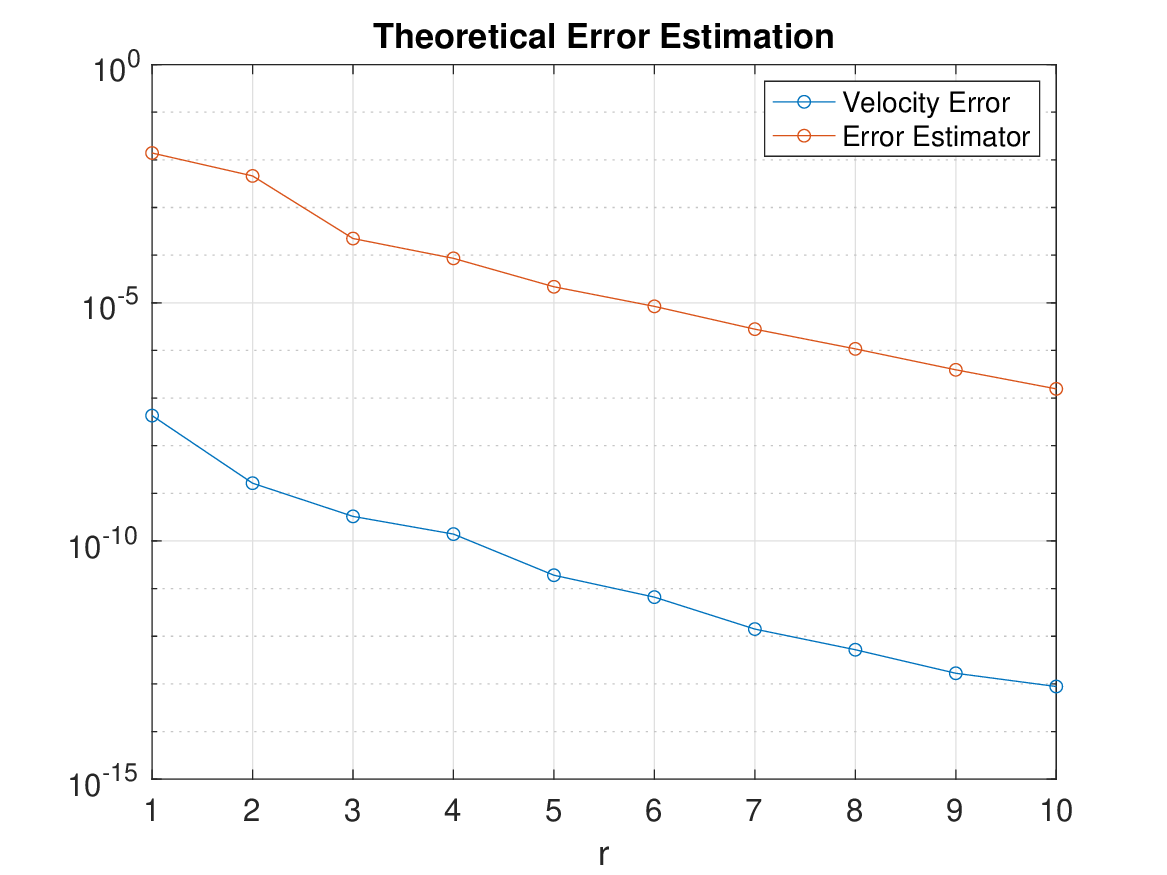}
		\end{subfigure}
		\begin{subfigure}[b]{0.4\textwidth}
			\includegraphics[width=\textwidth]{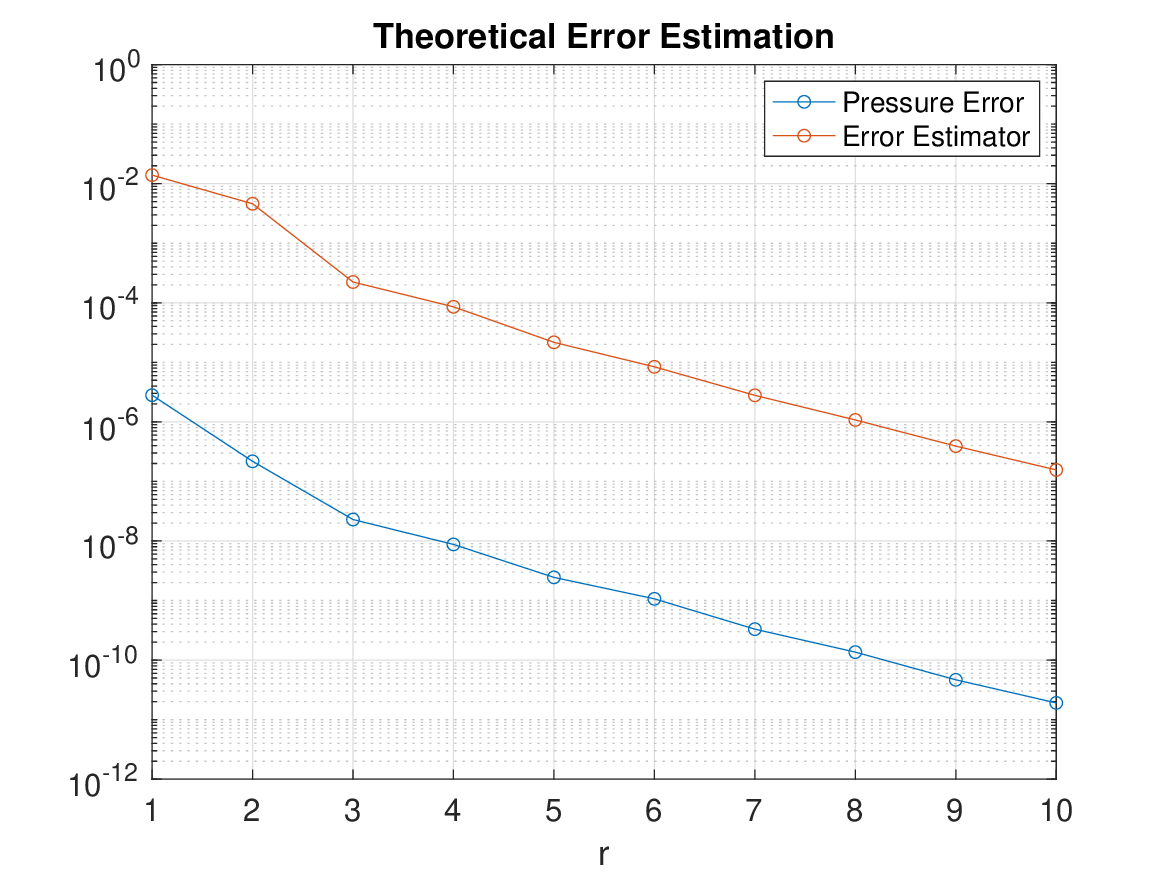}
		\end{subfigure}
		\caption{Comparison between the theoretical error estimator \eqref{rom_ea_th} and the velocity (left) and pressure (right) error depending on the number of modes $r$.}
		\label{fig:theoreticalnum}
	\end{figure}
	
	\subsection{Lid-driven cavity flow}
	The lid-driven cavity problem is a classical benchmark in computational fluid dynamics, widely studied for its rich flow structures and well-defined boundary conditions. This problem involves a square or rectangular domain where the fluid is driven by the motion of the upper boundary (lid) while the remaining walls remain stationary. In this work, we focus on a specific configuration where Dirichlet boundary conditions are imposed: the lid has a prescribed tangential velocity, and the side and bottom walls enforce no-slip conditions. Additionally, the right-hand side term in the governing equations is set to zero, i.e. $f = 0$, corresponding to the absence of external forces.
	
	Through numerical simulation, we analyze the velocity and pressure fields, highlighting the accuracy and stability of our method. The lid-driven cavity setup provides an excellent test case for evaluating the performance of numerical schemes, given its well-documented solutions and the presence of flow features such as vortices and shear layers.
	
	For this test, we have considered $\Omega = [0,1] \times [0,1]$ as computational domain. Additionally, for the ROM, both time and Reynolds number have been considered as parameters, resulting in a multiparametric ROM. The Reynolds number lies within the range $\mathcal{D} = [1000, 5000]$, where all solutions reach a steady-state regime.
	
	\subsubsection{Snapshots Generation for POD construction}
	
	The numerical method used for solving the problem is the one described in Section \eqref{ROMScheme}, moreover the spatial discretization used is $\mathbb{P}^2-\mathbb{P}^1$ FE for the pair velocity-pressure on a mesh that is refined towards the walls in both directions to accurately capture the unknown fields, using the hyperbolic tangent function 
	\cite{haferssas2018efficient, rubino2019efficient}:

	\begin{equation*}
		g(x) = 0.5 \left (1 + \dfrac{\tanh{(2(2x-1))}}{\tanh{(2)}} \right ).
	\end{equation*}
	We have considered a cavity mesh with a partition of $64^2$, which refers to a structured mesh with $64\times 64$ partition of the unit square (computational domain) in both $x,y$ directions.

	In the FOM simulations, an impulsive start is performed, with initial conditions set to zero for both velocity and pressure fields. The time step used is \( \Delta t = 5 \cdot 10^{-3} \) . Time integration is carried out using the incremental projection detailed in Section \eqref{FullyDiscrete}. To handle the convection non-linear Navier–Stokes term  \( (\bm{u} \cdot \nabla) \bm{\tilde{u}} \)  which appears in the momentum equation, we have considered the following extrapolation by means of Newton-Gregory backward polynomials \cite{Cellier} : \( ((2 \bm{u}^n - \bm{u}^{n-1}) \cdot \nabla ) \bm{\tilde{u}}^{n+1}\), $n \geq 1,$ in order to achieve a second-order accuracy in time. 
	
	We have considered a multiparametric ROM, as mentioned above, considering a physical parameter, the Reynolds number $Re$, and the time $t$. For the physical parameter we take 5 sample in a uniformly partition of the range $\mathcal{D} = [1000,5000].$ For the time variable, although we have considered solutions that reach a steady-state regime, we have taken the solution in the interval $[2,3) $, where the steady-state has not yet been achieved. Finally, we collect $M = 200 \cdot 5 = 1000$ snapshots for each unknown field. The POD modes are generated using $L^2-$norm for both velocity and pressure. To determine the dimension of the reduced space, Figure \eqref{fig:eigenvaluesNSE} illustrates the decay of the normalized singular values (left) and the neglected energy (right). From Figure \eqref{fig:eigenvaluesNSE}, we can observe that with $54$ modes, we neglect less than $10^{-6}$ of the energy for both velocity and pressure. 
	
	\begin{figure}[ht!]
		\centering
		\begin{subfigure}[b]{0.4\textwidth}
			\includegraphics[width=\textwidth]{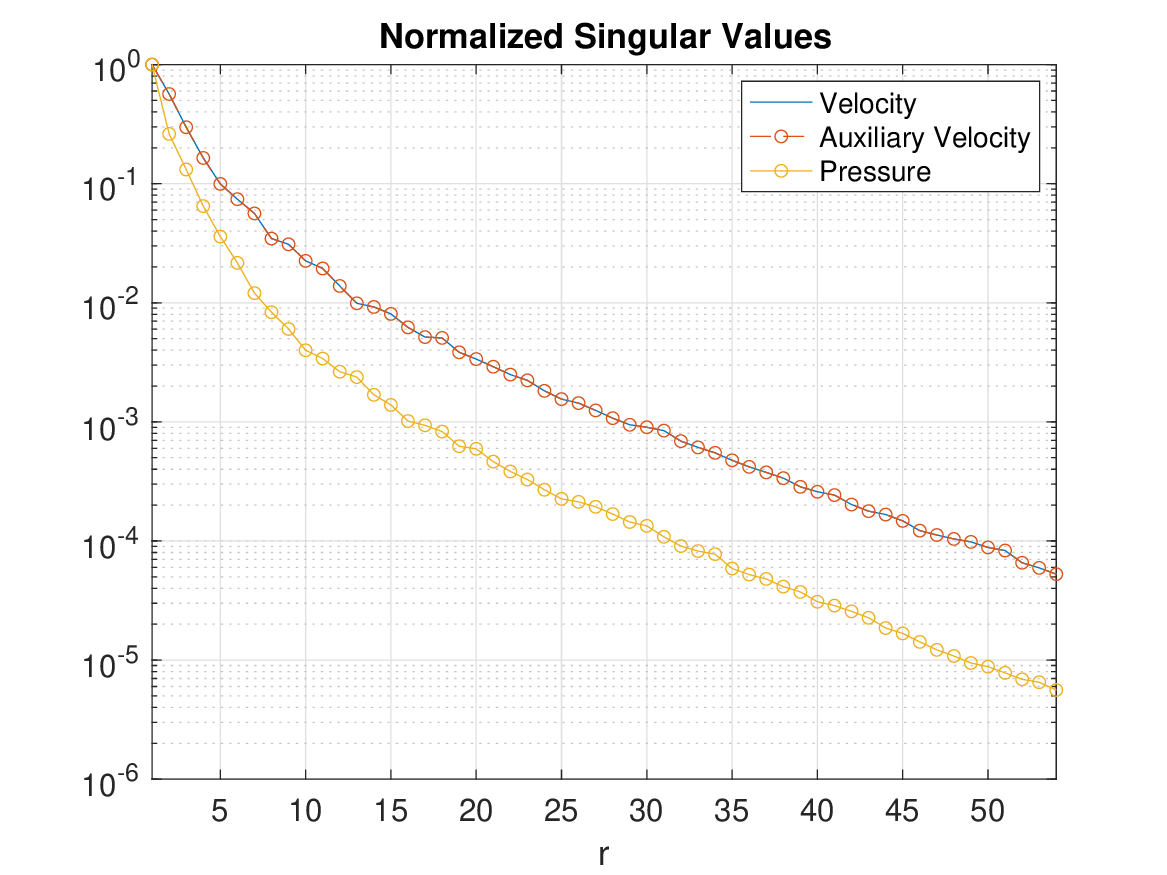}
		\end{subfigure}
		\begin{subfigure}[b]{0.4\textwidth}
			\includegraphics[width=\textwidth]{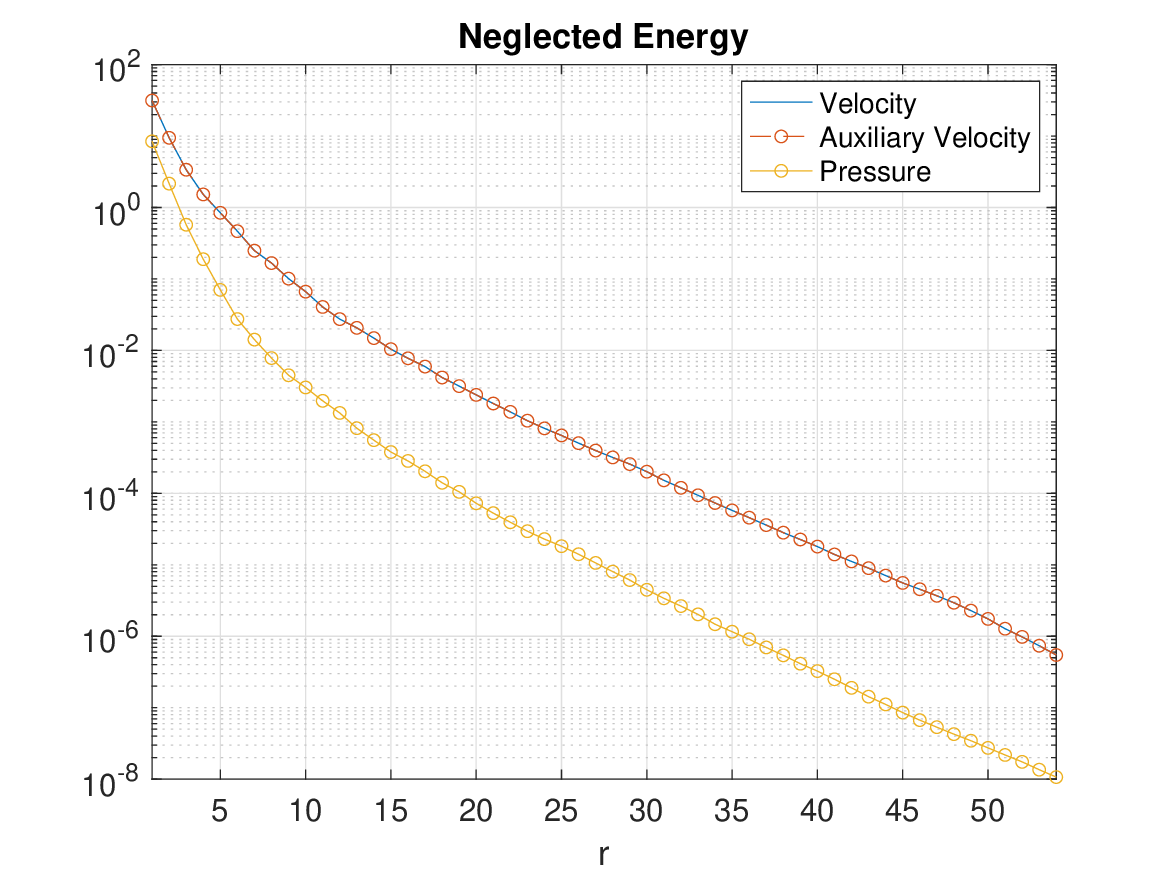}
		\end{subfigure}
		\caption{Decay of the normalized POD singular values (left) and neglected energy (right).}
		\label{fig:eigenvaluesNSE}
	\end{figure}

	\subsubsection{Numerical Results}
	Once the POD modes are generated, the ROM is constructed by using the same time discretization as the FOM. We have evaluated the ROM in the same time interval where the snapshots are taken, for different physical parameters values, and also extrapolating in time, i.e. evaluating the ROM over a longer time window than the one used to generate the data for its construction.
	
	Figure \eqref{fig:l2l2errorNSE} shows the $l^2(L^2)$ relative errors for both velocity (left) and pressure (right) for two different values of the physical parameter set, compared to the $l^2(L^2)$ relativer projection errors \eqref{projerror}, since they represent the best achievable error. 

	\begin{figure}[ht!]
		\centering
		\begin{subfigure}[b]{0.4\textwidth}
			\includegraphics[width=\textwidth]{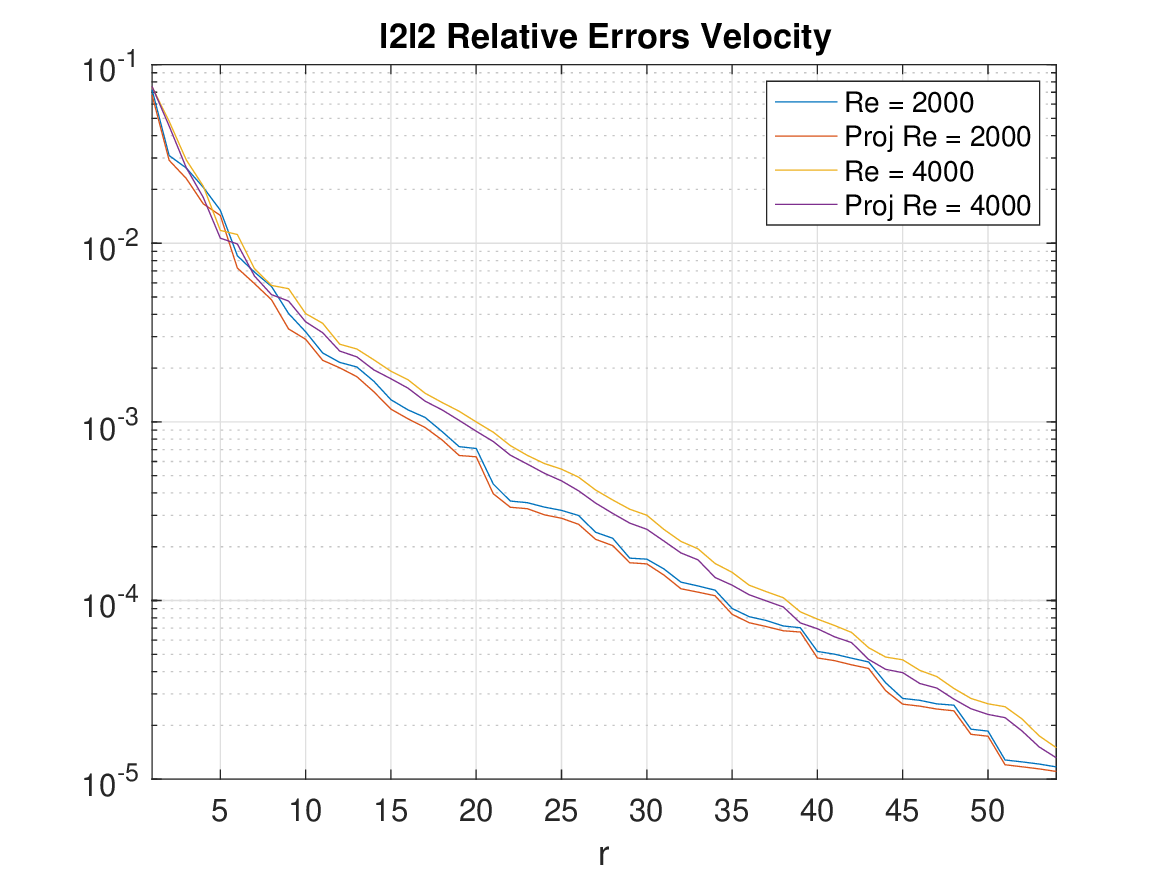}
		\end{subfigure}
		\begin{subfigure}[b]{0.4\textwidth}
			\includegraphics[width=\textwidth]{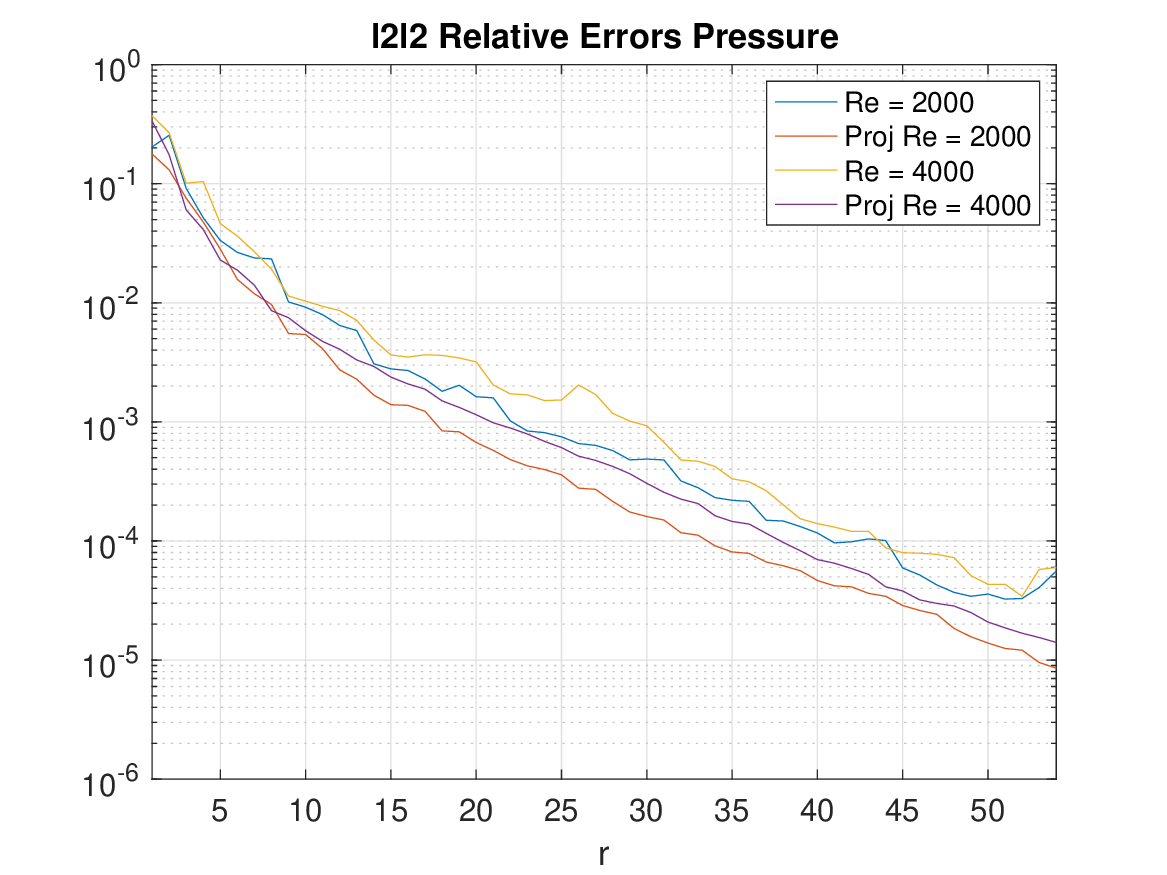}
		\end{subfigure}
		\caption{$\ell^2(L^2)$ relative errors for velocity (left) and pressure (right) depending on the number of POD modes $r.$}
		\label{fig:l2l2errorNSE}
	\end{figure}

	We can see in Figure \eqref{fig:l2l2errorNSE} that the $l^2(L^2)$ error exhibits the same behavior as its projection error, achieving a good precision with few modes.
	
	We also evaluate the performance of the ROM with respect to extrapolation in time. Using two samples of the Reynolds parameter $( Re \in \{ 2000, 4000 \})$, we analyze the ROM over the time interval $(2,4] $, while snapshots are collected in the time interval $(2,3]$. To assess the ROM's performance in this scenario, Figure 8 depicts the temporal evolution of the $L^2$ relative error in space for velocity (left) and pressure (right).

	\begin{figure}[ht!]
		\centering
		\begin{subfigure}[b]{0.4\textwidth}
			\includegraphics[width=\textwidth]{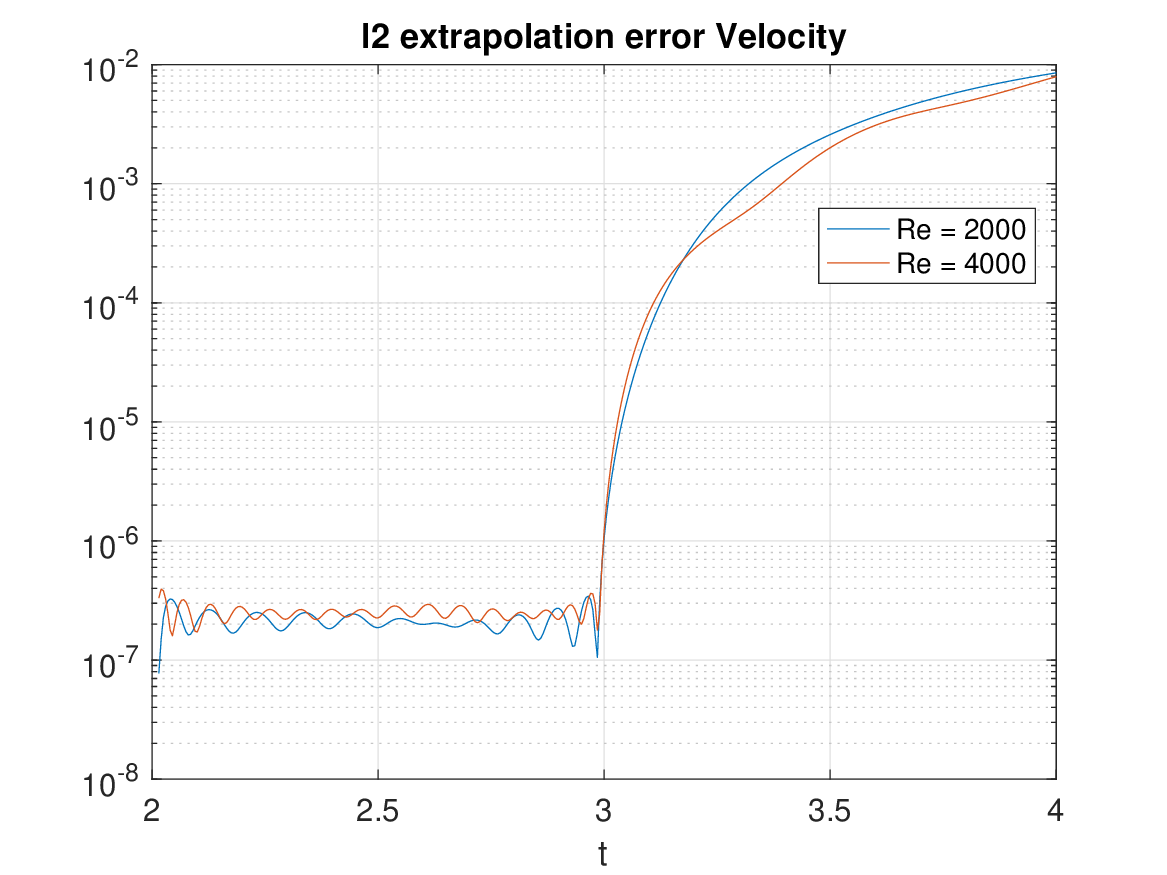}
		\end{subfigure}
		\begin{subfigure}[b]{0.4\textwidth}
			\includegraphics[width=\textwidth]{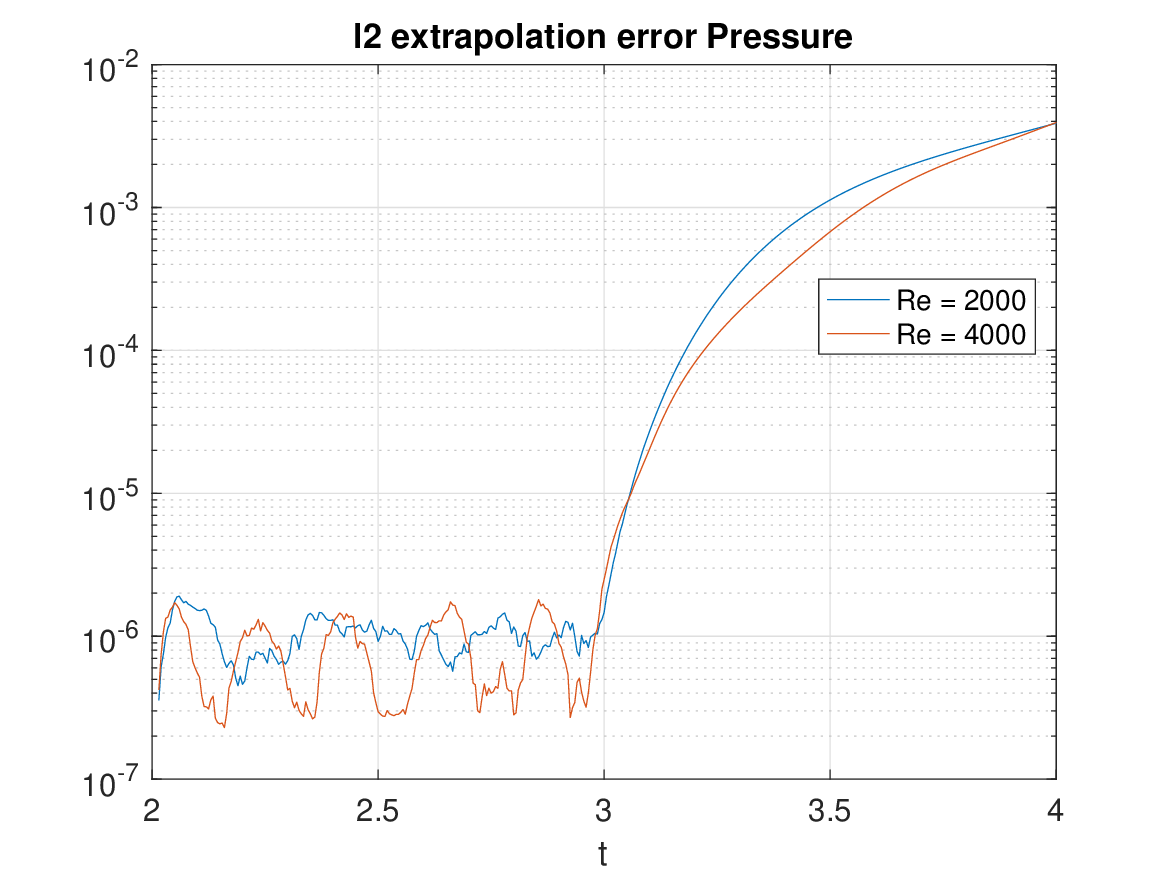}
		\end{subfigure}
		\caption{Temporal evolution of the $L^2$ relative error for velocity (left) and pressure (right) in the time interval $(2,4].$ }
		\label{fig:l2errorextrapolation}
	\end{figure}

	At the same time, to evaluate the performance of the ROM with respect to interpolation on the physical parameter, we select different values from the samples inside the range $\mathcal{D} = [1000,5000].$ Figure \eqref{fig:l2l2parameter} show the $l^2(L^2)$ relative error of velocity and pressure computed in the same time interval as for the snapshots, i.e. $[2,3).$
	
	In Figure \eqref{fig:l2l2parameter}, we can see that the maximum error for both velocity and pressure occurs at the mean value between each sample of the physical parameter, with the highest errors obtained for $Re=1500.$
	\begin{figure}[ht!]
		\centering
		\includegraphics[width=0.5\linewidth]{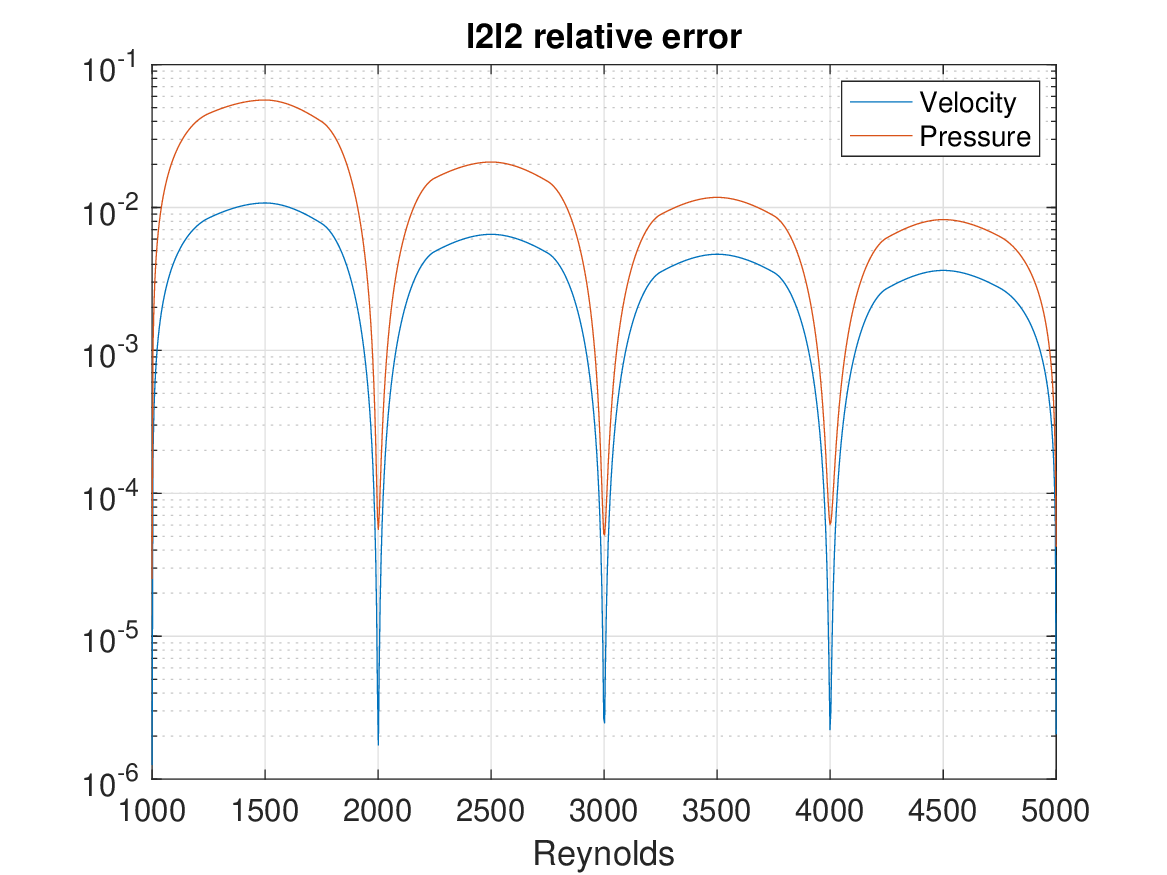}
		\caption{$\ell^2(L^2)$ relative error for velocity and pressure for different Reynolds numbers in the time interval $(2,3].$}
		\label{fig:l2l2parameter}
	\end{figure}
	\section{Conclusion}
	In this work, we have proposed a reduced-order model (ROM) for the incompressible Navier--Stokes equations based on an incremental projection scheme combined with Proper Orthogonal Decomposition (POD) techniques. The model employs BDF2 time discretization and finite element spatial approximation, achieving second-order accuracy in time.
	
	The reduced scheme effectively decouples velocity and pressure computations, while preserving stability and bypassing the inf-sup condition via a stabilized pressure formulation. We provided a detailed stability analysis and established error estimates, demonstrating that the velocity error in the semi-discrete norm is of order $\mathcal{O}(\delta t^2 + h^{l+1})$, and the splitting error remains of order $\mathcal{O}(\delta t^2)$.
	
	Numerical results validate the theoretical analysis. The benchmark Stokes problem confirmed the expected second-order convergence in time. In the POD-ROM simulations, very few modes were sufficient to attain high accuracy.
	
	The classical lid-driven cavity problem served as a further benchmark to evaluate the ROM under multiparametric conditions, including variations in Reynolds number and time. The ROM successfully captured key flow features, demonstrating both accuracy and computational efficiency, even when extrapolating in time and interpolating in parameter space.
	
	Overall, this framework presents a promising approach for reduced-order modeling of incompressible flows. Future research will focus on extending the method to the full nonlinear Navier--Stokes equations, adaptive mode selection strategies, and integration with data-driven or machine learning approaches to enhance ROM capabilities in more complex settings.
	
	\medskip
	
	\noindent {\bf Acknowledgments:} This work has been supported by the Spanish Government Project PID2021-123153OBC21 funded by MCIN/AEI/10.13039/501100011033/FEDER, UE.\\
	
	\clearpage
	\bibliographystyle{plain}
	\bibliography{reference}
	
\end{document}